\documentclass[10pt]{amsart}
\usepackage{amssymb,amstext,amsmath,amscd,amsthm,amsfonts,enumerate,latexsym,stmaryrd,multicol,geometry,graphicx,mathrsfs}
\usepackage{xcolor}
\definecolor{darkgreen}{rgb}{0.0, 0.6, 0.0}
\usepackage[colorlinks,citecolor=darkgreen,linkcolor=red]{hyperref}
\usepackage[all]{xy}
\newtheorem{thm}{Theorem}[section]
\newtheorem{lem}[thm]{Lemma}
\newtheorem{prop}[thm]{Proposition}
\newtheorem{cor}[thm]{Corollary}
\theoremstyle{definition}
\newtheorem{dfn}[thm]{Definition}
\newtheorem{ques}[thm]{Question}

\newtheorem{rem}[thm]{Remark}
\newtheorem{conv}[thm]{Convention}

\newtheorem{ex}[thm]{Example}

\theoremstyle{remark}

\newtheorem*{claim*}{Claim}

\numberwithin{equation}{thm}
\def\add{\operatorname{add}}
\def\ann{\operatorname{Ann}}
\def\ass{\operatorname{Ass}}
\def\Burch{\operatorname{Burch}}

\def\codepth{\operatorname{codepth}}
\def\codim{\operatorname{codim}}
\def\db{\operatorname{D^b}}
\def\depth{\operatorname{depth}}
\def\ds{\operatorname{D^{sg}}}
\def\dx{\operatorname{dx}}
\def\E{\operatorname{E}}
\def\e{\operatorname{e}}
\def\edim{\operatorname{edim}}
\def\Ext{\operatorname{Ext}}
\def\ge{\geqslant}
\def\grade{\operatorname{grade}}

\def\hh{\operatorname{h}}
\def\Hom{\operatorname{Hom}}
\def\height{\operatorname{ht}}
\def\I{\operatorname{I}}
\def\id{\operatorname{id}}
\def\image{\operatorname{Im}}
\def\k{\operatorname{K}}
\def\Ker{\operatorname{Ker}}
\def\le{\leqslant}
\def\level{\operatorname{level}}
\def\lhom{\operatorname{\underline{Hom}}}
\def\lmod{\operatorname{\underline{mod}}}
\def\lten{\otimes^\mathbf{L}}
\def\m{\mathfrak{m}}
\def\mod{\operatorname{mod}}
\def\N{\mathbb{N}}
\def\n{\mathfrak{n}}
\def\ocm{\operatorname{\Omega CM}}
\def\P{\operatorname{P}}
\def\p{\mathfrak{p}}
\def\pd{\operatorname{pd}}
\def\r{\operatorname{r}}
\def\rank{\operatorname{rank}}
\def\rhom{\operatorname{\mathbf{R}Hom}}
\def\s{\operatorname{s}}
\def\Sing{\operatorname{Sing}}
\def\soc{\operatorname{Soc}}
\def\spec{\operatorname{Spec}}
\def\sw{\operatorname{SW}}
\def\syz{\mathrm{\Omega}}
\def\T{\mathcal{T}}
\def\thick{\operatorname{thick}}
\def\Tor{\operatorname{Tor}}
\def\V{\operatorname{V}}
\def\X{\mathcal{X}}
\def\xx{\boldsymbol{x}}
\def\Y{\mathcal{Y}}
\def\yy{\boldsymbol{y}}
\def\Z{\mathbb{Z}}
\def\ZZ{\mathcal{Z}}
\begin{document}
\title{On the ubiquity of uniformly dominant local rings}
\author{Toshinori Kobayashi}
\address{School of Science and Technology, Meiji University, 1-1-1 Higashimita, Tamaku, Kawasaki 214-8571, Japan}
\email{tkobayashi@meiji.ac.jp}
\author{Ryo Takahashi}
\address{Graduate School of Mathematics, Nagoya University, Furocho, Chikusaku, Nagoya 464-8602, Japan}
\email{takahashi@math.nagoya-u.ac.jp}
\urladdr{https://www.math.nagoya-u.ac.jp/~takahashi/}
\thanks{Kobayashi was partly supported by JSPS Grant-in-Aid for Early-Career Scientists 25K17240.
Takahashi was partly supported by JSPS Grant-in-Aid for Scientific Research 23K03070}
\begin{abstract}
Let $R$ be a $d$-dimensional Cohen--Macaulay complete local ring with infinite residue field $k$.
The dominant index $\dx(R)$ is by definition the least number of extensions necessary to build $k$ in the singularity category $\ds(R)$ out of each nonzero object, up to finite direct sums, direct summands and shifts.
The local ring $R$ is called uniformly dominant if $\dx(R)$ is finite.
In this paper, we prove that $R$ is uniformly dominant with $\dx(R)\le6d+5$ if $R$ has codimension $2$ and is not a complete intersection.
Also, we show that $R$ is uniformly dominant with $\dx(R)\le d+1$ if $R$ is Burch, and with $\dx(R)\le d$ if $R$ is either a quasi-fiber product ring, or has multiplicity at most $5$ and is not Gorenstein.
A result on hypersurfaces by Ballard, Favero and Katzarkov is recovered, and results on Burch rings and quasi-fiber product rings by Takahashi are refined.
\end{abstract}
\maketitle
\tableofcontents
\section{Introduction}

Generation of a module from another by applying basic operations such as taking syzygies, direct summands and extensions has been playing a key role in the homological and categorical studies of modules over rings.
Takahashi \cite{stcm,thd} has established a classification of the thick subcategories of the singularity category of a hypersurface using generation of modules.
This does motivate us to seek for wider classes of rings than that of hypersurfaces to which similar module generation techniques apply.

In light of a result of Burch \cite{Bur}, Dao, Kobayashi and Takahashi \cite{burch} have introduced the notion of a {\em Burch ideal} of a local ring, and have shown that over the quotient ring of a regular local ring by a Burch ideal, every nonfree module generates the residue field by taking syzygies, direct summands and just one extension.
As a consequence, such a ring is Tor/Ext-friendly in the sense of Avramov, Iyengar, Nasseh and Sather-Wagstaff \cite{AINS}.
In particular, such a ring is G-regular in the sense of Takahashi \cite{greg} if it is not a hypersurface, and the Auslander--Reiten conjecture holds for such a ring.
A lot of ideals are known to be Burch ideals, including $\m$-primary integrally closed ideals and nonzero ideals of the form $\m I$ for some ideal $I$, where $\m$ is the maximal ideal.
However, a local ring defined by a Burch ideal must have depth zero. 
As an extension to the positive depth case, the notion of a {\em Burch ring} has been introduced also in \cite{burch}.
This is a local ring whose completion is a deformation of the quotient ring of a regular local ring by a Burch ideal.
A hypersurface is a typical example of a Burch ring.
It is straightforward that a Burch ring is Tor/Ext-friendly.
Dao and Eisenbud \cite{DE} have recently introduced the notion of a {\em Burch index}, and the Burch rings are characterized as the local rings with positive Burch index.
A local ring with Burch index greater than one exhibits a certain extremal property on syzygies of modules; see \cite{DE,DM} for details.
Many other works on Burch ideals and Burch rings have been done so far; see \cite{CK,Dao,DK,GS} for instance.

The notion of a {\em quasi-fiber product ring} in the sense of Freitas, Jorge P\'erez, R. Wiegand and S. Wiegand \cite{FJWW} has been investigated by several authors in recent years.
This notion is the same as that of a local ring with quasi-decomposable maximal ideal in the sense of Nasseh and Takahashi \cite{fiber}.
Over a quasi-fiber product ring, some syzygy of the residue field can be constructed from every module of infinite projective dimension by taking a direct summand of an extension of syzygies.
In particular, quasi-fiber product rings share many properties with Burch rings including Tor/Ext-friendliness.
One of the next expectations is thus to unify the theories of Burch rings and quasi-fiber product rings.

The notions of a {\em dominant local ring} and a {\em uniformly dominant local ring} have been introduced by Takahashi \cite{dlr,udim}.
A dominant local ring is a local ring whose residue field can be classically generated in the sense of Bondal and Van den Bergh \cite{BV} by any nonzero object in the singularity category.
The {\em dominant index} $\dx(R)$ of a local ring $R$ is the least number $n$ such that the residue field is generated in the singularity category by any nonzero object by taking finite direct sums, direct summands, shifts and at most $n$ extensions.
A local ring is called uniformly dominant if it has finite dominant index.
Thus a uniformly dominant local ring is dominant, while a dominant local ring is Tor/Ext-friendly.
Both a Burch ring and a quasi-fiber product ring are uniformly dominant.
The hierarchy of these classes of rings and some others is provided in Remark \ref{r312}.
A remarkable fact is that the thick subcategories of the singularity category are classified completely if the localization at each prime ideal is dominant.
We refer the reader to \cite{dlr,udim} for details.

A {\em Golod local ring} is a local ring defined by an extremal behavior of the Poincar\'e series of the residue field.
Basic facts on Golod rings are found in \cite{A}.
It is known that a Golod ring is Tor/Ext-friendly \cite{AINS,Jor}.
Typical examples of Golod local rings include local hypersurfaces and Cohen--Macaulay local rings with minimal multiplicity, which are also Burch rings.
In general, however, a Golod local ring is neither a Burch ring nor a quasi-fiber product ring; see \cite{Des,Les}.
These observations lead us to investigate the relationship between Golodness and dominance further.
It is thus quite natural to ask the following question.

\begin{ques}\label{35}
\begin{enumerate}[(1)]
\item
When is a given local ring dominant, or more strongly,  uniformly dominant?
\item
If a given local ring is uniformly dominant, then how big is its dominant index?
\item
Is every Golod local ring dominant, or more strongly, uniformly dominant?
\end{enumerate}
\end{ques}

For a local ring $R$, denote by $\dim R$, $\edim R$, $\codim R$, $\e(R)$, $\ell(R)$ and $\r(R)$ the (Krull) dimension, embedding dimension, (embedding) codimension, (Hilbert--Samuel) multiplicity, length, and type of $R$, respectively.
For a finitely generated $R$-module $M$, denote by $\nu(M)$ and $\P_M(t)$ the minimal number of generators and Poincar\'e series of $M$, respectively.
The following theorem is a summary of several special cases of the main results of this paper, which gives answers to the above Question \ref{35}.

\begin{thm} \label{t12}
Let $R$ be a Cohen--Macaulay complete local ring with infinite residue field $k$.
Put $d=\dim R$, $c=\codim R$ and $r=\r(R)$.
Then the following statements hold true.
\begin{enumerate}[\rm(1)]
\item
The local ring $R$ is uniformly dominant with $\dx(R)\le d$ in each of the following cases:
\begin{itemize}
\item
$R$ is a quasi-fiber product ring (e.g., $R$ has minimal multiplicity, i.e., $\e(R)=c+1$).
\item
$R$ is a non-Gorenstein ring with $\e(R)\le5$.
\item
$R$ is such that $\e(R)\le c+2$ and $\P_k(t)\ne(1+t)^d/(1-ct+rt^2)$.
\end{itemize}
\item
The local ring $R$ is uniformly dominant with $\dx(R)\le d+1$ in each of the following cases:
\begin{itemize}
\item
$R$ is a Burch ring (e.g., $R$ is a hypersurface).
\item
$R$ is a non-Gorenstein G-regular ring with $\e(R)\le 6$.
\item
$R$ is a non-Gorenstein ring with $c=2$ and $\e(R)\le 11$.
\end{itemize}
\item
Assume that the dimension $d$ is at most $2$ and the residue field $k$ is algebraically closed.
If the Cohen--Macaulay local ring $R$ has finite representation type, $R$ is uniformly dominant with $\dx(R)\le\max\{1,d\}$.
\item
Assume that the codimension $c$ is at most $2$.
Then $R$ is either a local complete intersection with $c=2$, or a uniformly dominant local ring such that $\dx(R)\le6d+5$.
\end{enumerate}
\end{thm}

In fact, we shall obtain in this paper more general versions of assertions (1) and (2) of Theorem \ref{t12}, which not only recover and extend the upper bound for $\dx(R)$ given by Ballard, Favero and Katzarkov \cite{BFK} when $R$ is a complete local hypersurface, but also refine the upper bounds for $\dx(R)$ given by Takahashi \cite{udim} when $R$ is a Burch ring and when $R$ is a quasi-fiber product ring.

A classical result due to Scheja \cite{Sch} shows that every Cohen--Macaulay local ring of codimension at most two is either a complete intersection or a Golod ring.
Theorem \ref{t12}(4) thus supports Question \ref{35}(3) in the affirmative, which is also shown in \cite{udim} under the additional assumption that the defining ideal is a monomial ideal.
Furthermore, Theorem \ref{t12} tells us that uniformly dominant local rings are ubiquitous.

Assertions (1) and (2) of Theorem \ref{t12} evaluate the dominant index of a local ring $R$ in the case where $R$ has small multiplicity.
This springs from our investigations of (artinian) local rings of small length via their defining ideals and some numerical invariants.
We summarize them as the following theorem.
As a study of commutative artinian rings, this theorem would be of independent interest.

\begin{thm} \label{t13}
Let $R$ be a non-Gorenstein local ring.
Then the following statements hold true.
\begin{enumerate}[\rm(1)]
\item
The local ring $R$ is a Burch ring provided that it satisfies one of the following four conditions.
\begin{itemize}
\item
$\ell(R) \le \r(R)+5$ and $\edim R=3 \le \r(R)$.
\item
$\ell(R)< 2\r(R)(\r(R)+1)$ (e.g., $R$ is non-Gorenstein with $\ell(R)\le 11$) and $\edim R=2$.
\item
$\ell(R)\le 5$.
\item 
$\ell(R)\le 6$ and $R$ is G-regular.
\end{itemize}
\item
Assume that $R$ is non-Burch with $\ell(R)=6$.
Let $R\cong S/I$ where $(S,\n)$ is a regular local ring and $I\subseteq\n^2$ an ideal.
Then $\edim R=3$, $\r(R)=2$, $\nu(I)=4$, and there exists an exact pair of zero-divisors in $R$.
\end{enumerate}
\end{thm}

As a corollary of Theorem \ref{t13}, one obtains the following result, which says that in the case of multiplicity at most six, most of the conditions which appear above turn out to be equivalent.
The assumption of having multiplicity at most six cannot be dropped; see Remark \ref{36} for details.

\begin{cor} \label{c14}
Let $R$ be a Cohen--Macaulay non-Gorenstein local ring with infinite residue field.
If $\e(R)\le6$, then conditions {\rm(2)}--{\rm(6)} on $R$ given below are equivalent, while {\rm(1)}--{\rm(8)} are equivalent if $\dim R=0$.
$$
\begin{array}{l}
{\rm(1)}\ \text{Burch},\quad
{\rm(2)}\ \text{uniformly dominant},\quad
{\rm(3)}\ \text{dominant},\quad
{\rm(4)}\ \text{Tor-friendly},\quad
{\rm(5)}\ \text{Ext-friendly},\\
{\rm(6)}\ \text{G-regular},\quad
{\rm(7)}\ \text{with no embedded deformation},\quad
{\rm(8)}\ \text{with no exact pair of zero-divisors}.
\end{array}
$$
\end{cor}

This paper is organized as follows.
Section 2 collects those definitions and basic facts on local rings which are used throughout the paper. 
Here we recall the definitions of a Burch ring and a Burch index.
In Section 3, we state some remarks on the dominant index and singularity category of a local ring.
The definition of a dominant index is based on the triangulated structure of the singularity category.
Proposition \ref{218} is regarded as an alternative definition of a dominant index based only on the structure of the module category.
In Section 4, we mainly study the Burchness and Burch indices of artinian local rings.
For an artinian local ring of embedding dimension two, multiple criteria for the ring to be Burch or to have Burch index at least two are obtained in terms of minimal free resolutions, lengths and types, in Propositions \ref{p210}, \ref{211}, Theorem \ref{t212}, and Corollary \ref{18}.
Then the first three cases of Theorem \ref{t13}(1) are completed.
In Section 5, we discuss exact pairs of zero-divisors.
We consider when an artinian local ring of length six has such a pair, and prove Theorem \ref{p311}, which is exactly the same as Theorem \ref{t13}(2).
We then investigate the relationships among Burchness, (uniform) dominance, and having no exact pairs of zero-divisors. 
The second assertion of Corollary \ref{c14} is proved in Corollary \ref{20}, which includes the last case of Theorem \ref{t13}(1). 
Section 6 is devoted to studies of upper bounds of the dominant index of a local ring.
In the case of Burch rings, the results obtained in the previous sections play essential roles.
The upper bounds due to Takahashi \cite[Corollary 5.5]{udim} are refined in Proposition \ref{t220}.
To deal with non-Burch rings in codimension two, we make use of Hilbert--Burch matrices.
Theorem \ref{t12} is obtained by combining Propositions \ref{p41}, \ref{14}, and Theorem \ref{t614}.
The first assertion of Corollary \ref{c14} is included in Proposition \ref{p41}.
In the final Section 7, we explore the generation of the residue field in the singularity category of a Gorenstein local ring with almost minimal multiplicity.
We derive a generation process of the residue field from some cyclic module in Theorem \ref{30}.
For further study of uniform dominance of Gorenstein local rings, we pose Question \ref{22}.

\section{Basic definitions and fundamental properties}

In this section, we present various definitions and properties which are necessary to state our results in later sections. 
We start with stating our convention, which is valid throughout the paper.

\begin{conv}
\begin{enumerate}[(1)]
\item
We assume that all rings are commutative noetherian rings, all modules are finitely generated, all complexes are cochain complexes, and all subcategories are strictly full.
\item
Unless otherwise specified, let $R$ be a local ring with maximal ideal $\m$ and residue field $k$.
\item
Set $(-)^\ast=\Hom_R(-,R)$, and denote by $\widehat R$ the ($\m$-adic) completion of $R$, by $\soc R$ the socle of $R$, by $\e(R)$ the (Hilbert--Samuel) multiplicity of $R$, and by $\mod R$ the category of (finitely generated) $R$-modules.
\item
For an $R$-module $M$ we denote by $\ell_R(M)$ the length of $M$, and by $\pd_RM$ (resp. $\id_RM$) the projective (resp. injective) dimension of $M$.
A regular sequence on an $R$-module $M$ is simply called an $M$-sequence.
\item
We may omit a subscript or a superscript if it is clear from the context.
\end{enumerate}
\end{conv}

We recall several basic notions on modules over a local ring.

\begin{dfn}
Let $(R,\m,k)$ be a local ring, $M$ an $R$-module, and $n$ an integer.
\begin{enumerate}[(1)]
\item
For $n\ge 1$ we define the {\em $n$th syzygy} of $M$, denoted $\syz^nM$, to be the image of the $n$th differential map in a minimal free resolution of $M$.
We set $\syz^0 M=M$ and $\syz M=\syz^1M$.
It is uniquely determined up to isomorphism, since so is a minimal free resolution of $M$.
\item
For a homomorphism $\phi$ of free $R$-modules, let $\I_n(\phi)$ stand for the ideal of $R$ generated by the $n$-minors of a representation matrix of $\phi$.
For an $R$-module $M$, take a minimal free presentation $F_1\xrightarrow{\phi}F_0\to M\to0$.
We then set $\I_n(M)=\I_n(\phi)$.
We refer the reader to \cite[\S1.4]{BH} for the details.
\item
For $n\ge0$ we denote by $\beta_n^R(M)$ and $\mu_R^n(M)$ the {\em $n$th Betti number} and {\em $n$th Bass number} of $M$, respectively, namely, 
$\beta_n^R(M)=\dim_k\Tor_n^R(M,k)=\dim_k\Ext_R^n(k,M)$ and $\mu_R^n(M)=\dim_k\Ext_R^n(k,M)$.
\item
We put $\nu_R(M)=\beta_0^R(M)$ and $\r_R(M)=\mu^u_R(M)$ where $u=\depth_RM$.
These integers are respectively called the {\em minimal number of generators} and {\em type} of $M$.
Note that $\beta_n^R(M)=\nu_R(\syz^nM)$ for $n\ge0$.
\item
We denote by $\P_M(t)$ the {\em Poincar\'e series} of $M$, that is to say, $\P_M(t)=\sum_{i\ge0}\beta_i(M)t^i\in\Z[\![t]\!]$.
\end{enumerate}
\end{dfn}

Next we recall some terminology concerning a local ring.

\begin{dfn}
\begin{enumerate}[(1)]
\item
For a local ring $(R,\m,k)$ we denote by $\edim R$, $\codim R$ and $\codepth R$ the {\em embedding dimension}, {\em (embedding) codimension} and {\em (embedding) codepth} of $R$, that is, $\edim R=\nu_R(\m)=\dim_k\m/\m^2$, $\codim R=\edim R-\dim R$ and $\codepth R=\edim R-\depth R$.
\item
In the case where $R$ is Cohen--Macaulay, $R$ is said to have \emph{minimal multiplicity} if $\e(R)=\codim R+1$.
\item
We say that $R$ is a {\em hypersurface} if there is an inequality $\codepth R\le1$.
This condition is equivalent to saying that there exist a regular local ring $S$ and an element $f\in S$ such that $\widehat R\cong S/(f)$; see \cite[\S5.1]{A}.
\item
We denote by $\Sing R$ the set of prime ideals $\p$ of $R$ such that the localization $R_\p$ of $R$ at $\p$ is not a regular local ring, and call it the \emph{singular locus} of $R$.
\item
We say that a local ring $(R,\m)$ has an \emph{isolated singularity} if the localization $R_\p$ is a regular local ring for any prime ideal $\p$ other than $\m$, or equivalently, if $\Sing R\subseteq\{\m\}$.
\end{enumerate}
\end{dfn}

Now we recall the definitions of a Burch ring and a Burch index, and basic facts on these notions.

\begin{dfn}
\begin{enumerate}[\rm(1)]
\item
An ideal $I$ of $R$ is called {\em Burch} if $\m(I:\m)\ne\m I$ holds.
\item
We say that $R$ is {\em Burch} if there exist a maximal $\widehat R$-sequence $\xx=x_1,\dots,x_n$, a regular local ring $S$ and a Burch ideal $I$ of $S$ such that $\widehat R/(\xx)\cong S/I$.
\item
Suppose that $R$ has depth zero.
Choose a regular local ring $(S,\n)$ and an ideal $I$ of $S$ with $\widehat{R}\cong S/I$.
Note that $\n^2\subseteq\n I:(I:\n)\subseteq\n$.
Assume that $R$ is not a field.
We put $\Burch (R)=\ell_S(\n/(\n I:(I:\n)))$ and call it the \emph{Burch index} of $R$.
This is independent of the choice of $S,I$; see \cite[Theorem 2.3]{DE}.
\end{enumerate}
\end{dfn}

\begin{rem}\label{r24}
\begin{enumerate}[(1)]
\item
If the maximal ideal $\m$ is such that $\m^2=0$, then $R$ is a Burch ring; see \cite[Example 2.2]{burch}.
\item
A Gorenstein local ring $R$ is Burch if and only if it is a hypersurface by \cite[Proposotion 5.1]{burch}.
\item
If $R$ is Cohen--Macaulay with minimal multiplicity and $k$ is infinite, $R$ is Burch by \cite[Proposition 5.2]{burch}.
\item 
Recall that an ideal $I$ of $R$ is called \emph{weakly $\m$-full} if there exists an ideal $J$ of $R$ such that the equality $I=J:\m$.
A weakly $\m$-full ideal $I$ of $R$ with $\depth R/I=0$ is a Burch ideal of $R$; see \cite[Corollary 2.4]{burch}.
\item
One has $\Burch(R)>0$ if and only if $R$ is a Burch ring.
This is straightforward from the definition.
\end{enumerate}
\end{rem}

Finally, we recall fiber products of local rings and some of their fundamental properties.

\begin{dfn}
\begin{enumerate}[(1)]
\item
Let $(R_1,\m_1)$ and $(R_2,\m_2)$ be local rings with common residue field $k$.
Let $\pi_i\colon R_i \to k$ be the canonical surjection with $i=1,2$.
The \emph{fiber product} $R:=R_1\times_kR_2$ of $R_1$ and $R_2$ over $k$ is defined as the subring of the direct product ring $R_1\times R_2$ consiting of elements $(a,b)$ with $\pi_1(a)=\pi_2(b)$.
\item
We say that a local ring $(R,\m,k)$ is a \emph{nontrivial fiber product (over $k$)} provided that there exist two local rings $(R_1,\m_1,k)$ and $(R_2,\m_2,k)$ such that $R\cong R_1\times_kR_2$ and $R_1\ne k\ne R_2$.
\end{enumerate}
\end{dfn}

\begin{rem}\label{33}
\begin{enumerate}[(1)]
\item
Let $(R_1,\m_1),(R_2,\m_2)$ be local rings with common residue field $k$.
Let $\rho_i\colon R \to R_i$ be the projections with $i=1,2$.
Then $I_1:=\Ker \rho_2 \cong \m_1$, $I_2:=\Ker \rho_1 \cong \m_2$, and $R$ is a local ring with maximal ideal $\m:=I_1+I_2$.
As $I_1\cap I_2=0$, there is a direct sum decomposition $\m=I_1\oplus I_2$ of $R$-submodules.
\item
Let $(R,\m,k)$ be a local ring.
As we just saw in (1), if $R$ is a nontrivial fiber product, then the maximal ideal $\m$ is decomposable as an $R$-module.
Conversely, if $\m=I\oplus J$ for some ideals $I$ and $J$ of $R$, then $R\cong R/I\times_k R/J$, whence $R$ is a nontrivial fiber product if moreover $I,J$ are nonzero.
\item
When $R=S\times_kT$ is a nontrivial fiber product, $\depth R=\min\{1, \depth S,\depth T\}$ by \cite[Remark 3.1]{CSV}.
\item
By (2) and \cite[Proposition 6.15]{burch}, a nontrivial fiber product $R_1\times_k R_2$ is a Burch ring if and only if either of the following two conditions is satisfied.\\
(i) One has $\min\{\depth R_1,\depth R_2\}\ge 1$.\quad
(ii) Either $R_1$ or $R_2$ is a Burch ring of depth zero.
\end{enumerate}
\end{rem}

\section{Dominant local rings and dominant indices}

In this section, we recall the definitions of a dominant local ring, uniformly dominant local ring and the dominant index of a local ring.
The main purpose of this section is to provide a module-theoretic description of a dominant index.
First of all, we recall the definition of a singularity category.

\begin{dfn}
The {\em singularity category} $\ds(R)$ of $R$ is defined to be the Verdier quotient of the bounded derived category $\db(\mod R)$ of $\mod R$ by perfect complexes over $R$.
\end{dfn}

By definition, $\ds(R)$ is a triangulated category.
This category has been introduced by Buchweitz \cite{B}.
If the local ring $R$ is Gorenstein, then $\ds(R)$ is equivalent as a triangulated category to the stable category of maximal Cohen--Macaulay $R$-modules; see \cite[Theorem 4.4.1]{B}.
There are a lot of works due to Orlov concerning singularity categories in relation to homological mirror symmetry; see \cite{O1,O2,O3} for instance.
The following are basic facts on singularity categories; for the proofs, we refer the reader to \cite[Lemma 2.4]{sing}.

\begin{rem}\label{21}
\begin{enumerate}[\rm(1)]
\item
For all $X\in\ds(R)$ there exist $M\in\mod R$ and $n\in\Z$ such that $X\cong M[n]$ in $\ds(R)$.
This isomorphism implies that $X$ is a nonzero object of $\ds(R)$ if and only if $\pd_RM=\infty$.
\item
Let $M$ be an $R$-module and $n\ge0$ an integer.
Then $M\cong(\syz^nM)[n]$ in $\ds(R)$.
\end{enumerate}
\end{rem}

Next we recall some notions on arbitrary triangulated categories.
The notions of thick subcategories and levels have been introduced respectively by Verdier \cite{V} and Avramov, Buchweitz, Iyengar and Miller \cite{ABIM}.

\begin{dfn}
Let $\T$ be a triangulated category.
\begin{enumerate}[(1)]
\item
A {\em thick subcategory} of $\T$ is defined as a triangulated subcategory of $\T$ closed under direct summands.
\item
Let $T$ be an object of $\T$.
The {\em thick closure} of $T$ in $\T$, which is denoted by $\thick_\T T$, is by definition the smallest thick subcategory of $\T$ containing $T$.
\item
Let $\X$ and $\Y$ be subcategories of $\T$.
We denote by $\langle\X\rangle$ the smallest subcategory of $\T$ which contains $\X$ and is closed under finite direct sums, direct summands and shifts.
We denote by $\X\ast\Y$ the subcategory of $\T$ consisting of objects $T\in\T$ which fits into an exact triangle $X\to T\to Y\rightsquigarrow$ in $\T$ such that $X\in\X$ and $Y\in\Y$.
We set $\langle\X\rangle_0^\T=0$, and $\langle\X\rangle_n^\T=\langle\langle\X\rangle_{n-1}^\T\ast\langle\X\rangle\rangle$ for $n\ge1$.
\item
For two objects $X$ and $Y$ of $\T$, we denote by $\level_\T^X(Y)$ the {\em level} of $Y$ with respect to $X$, which is defined to be the infimum of integers $n\ge-1$ such that $Y$ belongs to $\langle X\rangle_{n+1}$.
\end{enumerate}
\end{dfn}

Now, let us recall the notions of dominance, uniform dominance and dominant indices.

\begin{dfn}
\begin{enumerate}[(1)]
\item
We say that $R$ is {\em dominant} if $k\in\thick_{\ds(R)}X$ for all nonzero objects $X$ of $\ds(R)$.
\item
The {\em dominant index} of $R$ is defined by the following equalities:
$$
\begin{array}{l}
\dx(R)=\sup\{\level_{\ds(R)}^X(k)\mid0\ne X\in\ds(R)\}\\
\phantom{\dx(R)}=\inf\{n\in\Z_{\ge-1}\mid\text{$k\in\langle X\rangle_{n+1}$ for all nonzero objects $X$ of $\ds(R)$}\}.
\end{array}
$$
We say that the local ring $R$ is {\em uniformly dominant} provided that $\dx(R)<\infty$.
\end{enumerate}
\end{dfn}

\begin{rem}
\begin{enumerate}[(1)]
\item
By definition, a uniformly dominant local ring is a dominant local ring.
\item
A local ring $R$ is regular if and only if it satisfies the equality $\dx(R)=-1$.
\end{enumerate}
\end{rem}

In the rest of this section, we shall interpret the dominant index of $R$ without using the singularity category of $R$.
For this, we recall the definition of a Spanier--Whitehead category of the stable category of $\mod R$ and the notion of generation in $\mod R$ introduced by Dao and Takahashi \cite{radius}, and establish a lemma on them.

\begin{dfn}
\begin{enumerate}[\rm(1)]
\item
We denote by $\lmod R$ the {\em stable category} of $\mod R$, which is defined as follows.
The objects are the same as those of $\mod R$, that is, the (finitely generated) $R$-modules.
For two $R$-modules $M$ and $N$ the hom-set is defined as $\lhom_R(M,N)=\Hom_R(M,N)/\P_R(M,N)$, where $\P_R(M,N)$ consists of the homomorphisms $f:M\to N$ which decomposes as $f:M\to F\to N$ for some free $R$-module $F$.
\item
We denote by $\sw(\lmod R)$ the {\em Spanier--Whitehead category} of $\lmod R$, which is defined as follows.
The objects are the pairs $(M,m)$ of objects $M\in\lmod R$ and integers $m$. 
For two objects $(M,m)$ and $(N,n)$ the hom-set is defined by $\Hom_{\sw(\lmod R)}((M,m),(N,n))=\varinjlim_{i\ge m,n}\lhom_R(\syz^{i-m}M,\syz^{j-n}N)$.
\item
Let $\X$ and $\Y$ be subcategories of $\mod R$.
We denote by $[\X]$ the smallest subcategory of $\mod R$ containing $R$ and the modules in $\X$, and closed under finite direct sums, direct summands and syzygies.
Denote by $\X\circ\Y$ the subcategory of $\mod R$ consisting of $R$-modules $M$ which fits into an exact sequence $0\to X\to M\to Y\to0$ of $R$-modules with $X\in\X$ and $Y\in\Y$.
Set $[\X]_0=0$, and $[\X]_n=[[\X]_{n-1}\circ[\X]]$ for $n\ge1$.
\end{enumerate}
\end{dfn}

\begin{lem}\label{31}
\begin{enumerate}[\rm(1)]
\item
The category $\sw(\lmod R)$ is triangulated.
One has a triangle equivalence $\sw(\lmod R)\cong\ds(R)$, which sends $(M,m)$ to the complex $(\cdots\to0\to M\to0\to\cdots)$ where $M$ sits in degree $-m$.
\item
Let $(M,m)$ and $(N,n)$ be two objects of $\sw(\lmod R)$, and let $h$ be a nonnegative integer.
If  the object $(M,m)$ belongs to $\langle(N,n)\rangle_{h+1}$, then the syzygy $\syz^rM$ belongs to $[N]_{h+1}$ for some nonnegative integer $r$.
\item
Let $M$ be an $R$-module and $h$ a nonnegative integer.
The following are equivalent.
\begin{enumerate}[\rm(a)]
\item
For every nonzero object $X\in\ds(R)$, the $R$-module $M$, as an object of $\ds(R)$, belongs to $\langle X\rangle_{h+1}$.
\item
For every nonzero object $(N,n)\in\sw(\lmod R)$, the object $(M,0)$ belongs to $\langle(N,n)\rangle_{h+1}$.
\item
For each $R$-module $N$ with $\pd N=\infty$, there exists an integer $r\ge0$ such that $\syz^rM$ belongs to $[N]_{h+1}$.
\end{enumerate}
\end{enumerate}
\end{lem}

\begin{proof}
(1) The assertion is a direct consequenece of  \cite[Definition 3.1 and Theorem 3.8]{Be}.
We also refer the reader to \cite[Definition 2.4 and Theorem 3.2(3)]{sw}.

(2) We use induction on $h$.
First, let $h=0$.
Then the object $(M,m)$ belongs to $\langle(N,n)\rangle$.
Hence, there exists an isomorphism $(M,m)\oplus(L,l)\cong\bigoplus_{i=1}^s(N,n+a_i)^{\oplus b_i}$ in $\sw(\lmod R)$.
Applying the shift functor $[-t]$ for $t\gg0$, we have an isomorphism $(M,m-t)\oplus(L,l-t)\cong\bigoplus_{i=1}^s(N,n+a_i-t)^{\oplus b_i}$ so that the integers $m-t,l-t,n+a_i-t$ are all nonpositive.
We get an isomorphism $(\syz^{t-m}M,0)\oplus(\syz^{t-l}L,0)\cong\bigoplus_{i=1}^s(\syz^{t-n-a_i}N,0)^{\oplus b_i}$ in $\sw(\lmod R)$, which induces an isomorphism $\syz^{u+t-m}M\oplus\syz^{u+t-l}L\cong\bigoplus_{i=1}^s(\syz^{u+t-n-a_i}N)^{\oplus b_i}$ in $\lmod R$ for some $u\ge0$; see \cite[Corollary 3.3(3)]{Be} (or \cite[Lemma 5.2(2)]{sw}).
It follows that the module $\syz^{u+t-m}M$ belongs to $[N]$.

Next, let $h\ge1$.
There exists an exact triangle $(A,a)\to(B,b)\to(C,c)\rightsquigarrow$ in $\sw(\lmod R$) with $(A,a)\in\langle(N,n)\rangle_h$ and $(C,c)\in\langle(N,n)\rangle$ such that $(M,m)$ is a direct summand of $(B,b)$.
The induction hypothesis and basis imply $\syz^pA\in[N]_h$ and $\syz^qC\in[N]$ for some $p,q\ge0$.
Applying \cite[Lemma 5.2(4)]{sw}, we get an exact sequence $0\to\syz^{v-a}A\to\syz^{v-b}B\to\syz^{v-c}C\to0$ of $R$-modules up to free summands for some integer $v\ge a,b,c$.
Also, there is an integer $w\ge b,m$ such that $\syz^{w-m}M$ is a direct summand of $\syz^{w-b}B$ in $\lmod R$ by \cite[Lemma 5.2(3)]{sw}.
Put $r=\max\{p,q\}$ and $x=\max\{v,w\}$.
We see that, up to free summands, there is an exact sequence $0\to\syz^{r+x-a}A\to\syz^{r+x-b}B\to\syz^{r+x-c}C\to0$ whose middle term $\syz^{r+x-b}B$ contains $\syz^{r+x-m}M$ as a direct summand.
Note that $\syz^{r+x-a}A\in[N]_h$ and $\syz^{r+x-c}C\in[N]$.
We observe that $\syz^{r+x-m}M$ belongs to $[N]_{h+1}$.

(3)
(a)$\Leftrightarrow$(b):
This equivalence is an immediate consequence of the triangle equivalence given in (1).

(c)$\Rightarrow$(b):
Let $(N,n)$ be a nonzero object of $\sw(\lmod R)$.
Then by (1) the complex $(\cdots\to0\to N\to0\to\cdots)$ with $N$ sitting in degree $-n$ is a nonzero object of $\ds(R)$.
Hence the $R$-module $N$ has infinite projective dimension.
By assumption, there exists an integer $r\ge0$ such that $\syz^rM$ belongs to $[N]_{h+1}$.
Using \cite[Lemma 5.2(1)]{sw}, we get $(M,-r)=(\syz^rM,0)\in\langle(N,0)\rangle_{h+1}=\langle(N,n)\rangle_{h+1}$.
Thus the object $(M,0)$ is in $\langle(N,n)\rangle_{h+1}$.

(b)$\Rightarrow$(c):
Let $N$ be an $R$-module with $\pd N=\infty$.
Then the complex $(\cdots\to0\to N\to0\to\cdots)$ with $N$ sitting in degree $-n$ is a nonzero object of $\ds(R)$.
By (1), the object $(N,n)$ of $\sw(\lmod R)$ is nonzero.
By assumption, the object $(M,0)$ is in $\langle(N,n)\rangle_{h+1}$.
It follows from (2) that $\syz^rM$ is in $[N]_{h+1}$ for some $r\ge0$.
\end{proof}

Applying the equivalence (a)$\Leftrightarrow$(c) in Lemma \ref{31}(3) to the residue field of a local ring $R$, we readily get the proposition below, which enables us to understand the dominant index $\dx(R)$ only in terms of $R$-modules.

\begin{prop} \label{218}
For a local ring $R$ with residue field $k$, the following equality holds true.
$$
\dx(R)=\inf\{n\in\Z_{\ge-1}\mid\text{for any $R$-module $M$ with $\pd M=\infty$, there exists $r\in\N$ with $\syz^rk\in[M]_{n+1}$}\}.
$$
In particular, the dominant index $\dx(R)$ is at most zero if and only if for any $R$-module $M$ with $\pd M=\infty$, there exist $r,s,n_0,\dots,n_s\in\N$ such that $\syz^r k$ is a direct summand of $\bigoplus_{i=0}^s\syz^iM^{\oplus n_i}$.
\end{prop}

\section{Sufficient conditions for Burchness}

In this section, we shall provide sufficient conditions for a local ring to have positive Burch index in terms of several numerical invariants such as Hilbert functions, types, and lengths.
We begin with recalling the definitions of the Hilbert function of a local ring and a stretched artinian local ring.

\begin{dfn}
\begin{enumerate}[(1)]
\item
For a local ring $(R,\m,k)$ and an integer $i\ge 0$, set $h_i=\hh_i(R)=\dim_k\m^i/\m^{i+1}$.
The function $\hh_R\colon \N \to \N$ given by $i \mapsto h_i$ is called the \emph{Hilbert function} of $R$.
Clearly, there exists $n\ge0$ such that $h_i=0$ for any $i>n$, and then we say that $(h_0,h_1,\dots,h_n)$ is the Hilbert function of $R$.
\item
The \emph{socle degree} of an artinian local ring $R$ is the largest integer $s=\s(R)$ such that $h_s\not=0$.
Note that $\m^s\subseteq \soc R$.
We say that $R$ is \emph{stretched} if $h_2\le 1$.
In this case, we have $s(R)=\ell(R)-e$, and $R$ has Hilbert function $(1,e, 1,\dots, 1)$ with $e=\edim R$ by Macaulay's theorem \cite[Theorem 4.2.10]{BH}.
\end{enumerate}
\end{dfn}

We give criteria for an ideal of a regular local ring to be Burch, which are used several times later.

\begin{lem} \label{l25}
Let $R$ be a regular local ring, and let $I$ be a nonzero ideal of $R$.
\begin{enumerate}[\rm(1)]
\item
If $I$ is a Burch ideal of $R$, then the local ring $R/I$ has depth zero.
\item
Suppose that $R/I$ has depth zero.
Then the following four conditions are equivalent.\\
{\rm(a)} $I$ is Burch.\qquad
{\rm(b)} $R/I$ is Burch.\qquad
{\rm(c)} $\r(R/\m I)<\r(R/I)+\nu(I)$.\qquad
{\rm(d)} $\nu(I:\m)<\r(R/I)+\nu(I)$.
\item
Suppose that $R/I$ is an artinian stretched local ring with $e:=\edim R/I=\edim R$.
Then one has that $\nu(I)=\binom{e+1}{2}-1$ if $\r(R/I)<e$, and $\nu(I)=\binom{e+1}{2}$ if $\r(R/I)=e$.
\end{enumerate}
\end{lem}

\begin{proof}
The first assertion is immedate from the definition of a Burch ideal.
The third assertion follows from \cite[Theorem 2]{Sa} combined with \cite[Theorem 2.3.2]{BH}.
It remains to show the second assertion.

(a)$\Rightarrow$(b):
Observe that $I\widehat{R}$ is a Burch ideal of $\widehat{R}$, and hence $R/I$ is a Burch ring.

(b)$\Rightarrow$(a):
If $R/I$ is a field, then $I=\m$ and $\m(I:\m)=\m\ne\m^2=\m I$, whence $I$ is Burch.
If $R/I$ is not a field, then the assertion follows from \cite[Lemma 2.11]{burch}.

(a)$\Leftrightarrow$(c):
As $I\ne0$ and $I/\m I\subseteq R/\m I$, we have $\depth R/\m I=0$.
We have $\r(R/\m I)=\ell((\m I:\m)/\m I)\le\ell((I:\m)/\m I)=\ell((I:\m)/I)+\ell(I/\m I)=\r(R/I)+\nu(I)$.
Thus the assertion follows by \cite[Proposition 2.3]{burch}.

(a)$\Leftrightarrow$(d):
If $I$ is a Burch ideal of $R$, then it follows by definition that $R/I$ has depth zero.
The natural exact sequence $0\to I/\m I\to(I:\m)/\m I\to(I:\m)/I\to0$ shows that $\ell((I:\m)/\m I)=\ell((I:\m)/I)+\ell(I/\m I)=\r(R/I)+\nu(I)<\infty$.
Therefore, $I$ is a Burch ideal of $R$ if and only if $\m I\subsetneq\m(I:\m)$ holds, if and only if $\ell((I:\m)/\m(I:\m))<\ell((I:\m)/\m I)$ holds, if and only if $\nu(I:\m)<\r(R/I)+\nu(I)$ holds.
\end{proof}

The following proposition contains sufficient conditions for a local ring to be both a Burch ring and a fiber product.
The implication (c)$\Rightarrow$(d) in the proposition should be well-known to experts, but we do not know a reference.
It can be shown as a consequence of the implications (a)$\Leftarrow$(b)$\Leftarrow$(c), so we include a proof.

\begin{prop} \label{p210}
\begin{enumerate}[\rm(1)]
\item
Let $(R,\m)$ be a local ring.
The implications {\rm(a)}$\Leftarrow${\rm(b)}$\Leftarrow${\rm(c)}$\Rightarrow${\rm(d)} hold true, where
\begin{enumerate}[\rm(a)]
\item
The local ring $R$ is both a Burch ring and a nontrivial fiber product,
\item
One has $\soc R\not\subseteq \m^2$ and $\edim R\ge 2$,
\item
The local ring $R$ is artinian, stretched and non-Gorenstein,
\item
There is an inequality $\r(R)\le\edim R$.
\end{enumerate}
\item
Suppose that an artinian local ring $R$ satisfies $\edim R=3 \le \r(R)$ and $\ell(R) \le \r(R)+5$.
Then $R$ is Burch.
\end{enumerate}
\end{prop}

\begin{proof}
(1)
(a)$\Leftarrow$(b):
As $\soc R\not\subseteq \m^2$, we find a nonzero composition $k \to \m \to \m/\m^2$, which deduces $\m\cong k\oplus X$ for some $R$-module $X$.
The assumption $\edim R\ge 2$ guarantees that $X$ is nonzero.
Take the corresponding direct sum decomposition $\m=I\oplus J$ of ideals with $I\cong k$ and $J\cong X$.
We have $R\cong R/I\times_k R/J$.
Since the maximal ideal of $R/J$ is isomorphic to $k$, the local ring $R/J$ is an artinian hypersurface which is not a field.
In particular, $R$ is both a nontrivial fiber product and a Burch ring by Remarks \ref{r24} and \ref{33}.

(b)$\Leftarrow$(c):
Let $t\ge1$ be the largest integer such that $\soc R \subseteq \m^t$.
The induced map $\varphi\colon\soc R \to \m^t/\m^{t+1}$ is nonzero.
Assume that $t\ge 2$.
Since $R$ is stretched, we have $\ell(\m^t/\m^{t+1})=1$, which forces $\varphi$ to be surjective.
It follows that $\m^t=\soc R+\m\m^{t}$.
Nakayama's lemma implies that $\soc R=\m^t$.
We get $\r(R)=\nu(\soc R)=1$, which contradicts the assumption that $R$ is not Gorenstein.
Thus we have $t=1$, so that $\soc R\nsubseteq\m^2$.

(c)$\Rightarrow$(d):
We use induction on $\edim R$.
We may assume that $\edim R\ge2$.
As we have shown (c) implies (b) and (b) implies (a), the local ring $R$ is a nontrivial fiber product over $k$.
There are local rings $S,T$ with common residue field $k$ such that $R\cong S\times_k T$.
It is observed that $\r(R)=\r(S)+\r(T)$, $\edim R=\edim S+\edim T$, and $\hh_i(R)=\hh_i(S)+\hh_i(T)$ for any $i\ge 1$.
As $R$ is stretched, so are both $S$ and $T$.
Since $\edim S$ and $\edim T$ are less than $\edim R$, the induction hypothesis yields $\r(R)=\r(S)+\r(T)\le \edim S+\edim T=\edim R$.

(2) Set $\overline R=R/\soc R$.
We have $\ell(\overline R)=\ell(R)-\r(R)\le 5$.
In view of (1), we may assume $\soc R \subseteq \m^2$.
As $\edim R=3$, the Hilbert function of $\overline R$ is either $(1,3)$ or $(1,3,1)$.
Hence $\overline R$ is stretched.
By Cohen's structure theorem, there exist a regular local ring $(S,\n,k)$ of dimension three and an ideal $I$ of $S$ such that $R\cong S/I$.
Then $\overline R\cong S/(I:\n)$.
Since $R$ is not a complete intersection, we must have $\nu(I)>3$.
By (1) we get $\r(\overline R)\le3$.
Lemma \ref{l25}(3) shows $\nu(I:\n)\le\binom{3+1}{2}=6\le\r(R)+3< \r(R)+\nu(I)$, and $R$ is Burch by Lemma \ref{l25}(2).
\end{proof}

The next lemma gives an inequality of a Burch index and sufficient conditions for the equality to hold.

\begin{lem}\label{l217}
Let $(R,\m,k)$ be a local ring of depth zero.
Assume that $R$ is not a field.
The following hold.
\begin{enumerate}[\rm(1)]
\item
There is an inequality $\Burch(R)\le \edim R$.
\item
Suppose that there exists a regular local ring ring $(S,\n)$ such that $R\cong S/I$.
Then one has the equality $\Burch(R)=\edim R$, provided that the ideal $I$ satisfies one of the following conditions.
\begin{enumerate}[\rm(a)]
\item
There exists an ideal $J$ of $S$ such that $0\not=J\subseteq \n$ and $I=\n J$.
\item
There exists an ideal $J \subseteq \n^2$ such that $S/J$ is an artinian Gorenstein ring and $I=J:\n$.
\item
There exists an integer $s\ge2$ such that $I:\n\not\subseteq \n^s \supseteq I$.
\end{enumerate}
\end{enumerate}
\end{lem}

\begin{proof}
Assertions (1) and (2a) are due to \cite[Proposition 2.5]{DE}, while (2c) is a local analogue of \cite[Proposition 5.3]{DE}.
It remains to show (2b).
Since $S/J$ is an artinian Gorenstein ring, $\soc S/J=(\overline\sigma)\cong k$ for some $\sigma\in S$.
Also, taking the $S/J$-dual of the exact sequence $0\to\n/J\to S/J\to k\to0$, we get an isomorphism $S/I=(S/J)/\soc(S/J)\to\Hom_{S/J}(\n/J,S/J)$, which sends $\overline1\in S/I$ to the inclusion map $\theta:\n/J\to S/J$.

We claim that the equality $\n I:(I:\n)=\n^2$ holds.
Indeed, assume that there is an element $x\in\n I:(I:\n)$ with $x\notin\n^2$.
As $\overline x\in\n/J$ is part of a minimal system of generators, there is an epimorphism $\alpha:\n/J\to k$ with $\alpha(\overline x)=\overline1$.
There is a monomorphism $\beta:k\to S/J$ with $\beta(\overline1)=\overline\sigma$.
We find an element $y\in S$ with $\beta\alpha=y\theta$.
As $\overline\sigma\in\soc S/J$, we have $\n\beta=0$, so that $\n y\theta=\n\beta\alpha=0$.
We get $\n y\subseteq \ann_S(\theta)=I$, and $y\in I:\n$.
The choice of $x$ implies $xy\in\n I\subseteq J$.
We have $\overline\sigma=\beta\alpha(\overline x)=y\theta(\overline x)=\overline{xy}=\overline0$ in $S/J$, which contradicts the fact that $0\ne\soc S/J=(\overline\sigma)$.
The claim follows.

Since $I\ne\n$, we have $I:\n\subseteq\n$.
Hence $I(I:\n)\subseteq\n I$, which means $I\subseteq\n I:(I:\n)$.
The claim implies $I\subseteq\n^2$.
Using the claim again, we get $\edim R=\edim S=\ell(\n/\n^2)=\ell(\n/\n I:(I:\n))=\Burch(R)$.
\end{proof}

We present simple criteria for Burch rings and ones with Burch index two in the case of embedding dimension two.
The first assertion is a straightforward modification of \cite[Theorem 6.2]{burch}.

\begin{prop}\label{211}
Let $(R,\m,k)$ be a regular local ring of dimension two.
Let $I$ be an $\m$-primary ideal of $R$.
Take a minimal free resolution $0\to F_2\xrightarrow{\partial}F_1\to F_0\to R/I\to0$ of the $R$-module $R/I$.
\begin{enumerate}[\rm(1)]
\item
The following four coniditions are equivalent.
\begin{enumerate}[\rm(a)]
\item The local ring $R/I$ is a Burch ring.
\item There exist a regular system of parameters $x,y$ of $R$ and an integer $n\ge1$ such that $\I_1(\partial)=(x,y^n)$.
\item The ideal $\I_1(\partial)$ is not contained in $\m^2$.
\item There is an inequality $\r(R/(I:\m))<2\r(R/I)$.
\end{enumerate}
\item
Assume that $I\not=\m$. 
The equality $\Burch(R/I)=2$ holds if and only if the equality $\I_1(\partial)=\m$ holds.
\end{enumerate}
\end{prop}

\begin{proof}
(1) By \cite[Theorem 6.2]{burch} conditions (a) and (b) are equivalent, while clearly, (b) implies (c). 
Assume that $J:=\I_1(\partial)$ is not contained in $\m^2$.
We can choose an element $x\in J$ which is not in $\m^2$, and an element $y\in\m$ such that $x,y$ is a system of generators of $\m$.
As $R/(x)$ is a discrete valuation ring with a uniformizer $\overline y$, the ideal $J/(x)$ of $R/(x)$ is generated by $\overline y^n$ for some $n>0$.
Hence $J=(x,y^n)$.
Thus (c) implies (b).

We will be done if we show that (a) and (d) are equivalent.
We have $\nu(I)=\rank F_1=\rank F_2+\rank F_0=\r(R/I)+1$.
Similarly, we have $\nu(I:\m)=\r(R/(I:\m))+1$ (note that this holds even if $I=\m$).
Lemma \ref{l25} shows that the Burchness of the local ring $R/I$ is equivalent to the inequality $\nu(I:\m)<\r(R/I)+\nu(I)$, which is equivalent to the inequality $\r(R/(I:\m))<2\r(R/I)$.

(2) To show the ``only if'' part, suppose $\Burch(R/I)=2$.
By (1) and Remark \ref{r24}(5), there exist a regular system of parameters $x,y$ of $R$ and an integer $n>0$ such that $\I_1(\partial)=(x,y^n)$.
As $\ell_R(\m/\m^2)=2$, we have $\m I:(I:\m)=\m^2$. Hence $x$ is not in $\m I:(I:\m)$.
By \cite[Proposition 3.2]{DE}, the $R$-module $k$ is a direct summand of $I/xI$.
Tensoring $R/(x)$ with $F_\bullet$ gives a minimal free presentation $0 \to F_2/xF_2 \xrightarrow{\overline{\partial}} F_1/xF_1 \to I/xI \to 0$ of the $R/(x)$-module $I/xI$.
Thus $\I_1(\overline{\partial})=\I_1(I/xI)\supseteq \I_1(k)=\m/(x)$.
We now get $\m=\I_1(\partial)+(x)=\I_1(\partial)$.

To show the ``if'' part, assume that $\m^2\subsetneq\m I:(I:\m)$.
Then there exists an element $x\in \m I:(I:\m)\subseteq\m$ such that $x$ is not in $\m^2$.
Tensoring $R/(x)$ with $F_\bullet$, we see that $\overline{\partial}$ gives a minimal free presentation of $I/xI$.
As $\I_1(\partial)=\m$, we have $\I_1(\overline\partial)=\m/(x)$.
Since $R/(x)$ is a discrete valuation ring, there exists an element $y\in\m$ such that $\m/(x)=(\overline y)$.
Thus we have elementary transformations $\overline{\partial} \cong
\left(\begin{smallmatrix}
\overline y & *\\
* & *
\end{smallmatrix}\right)
\cong
\left(\begin{smallmatrix}
\overline y & 0\\
0 & *
\end{smallmatrix}\right)$ of matrices over $R/(x)$.
It follows that $k$ is a direct summand of $I/xI$.
Using \cite[Proposition 3.2]{DE}, we have $x\not\in \m I:(I:\m)$.
This contradicts the choice of $x$, and we obtain $\m^2=\m I:(I:\m)$.
Therefore, $\Burch(R/I)=\ell(\m/\m^2)=2$.
\end{proof}

Examples of local rings with Burch index one are easily found by using Proposition \ref{211}.

\begin{ex}\label{ex34}
Consider the quotient $R=k[x,y]/(x^2,xy^2,y^4)$ of a polynomial ring over a field $k$.
Then $R$ is an artinian non-Gorenstein Burch ring with $\ell(R)=6$, $\r(R)=2$ and $\Burch(R)=1$ by \cite[Example 4.5]{DE}.

More generally than this, let $S$ be a regular local ring of dimension two with a regular system of parameters $x,y$.
Let $r\ge1$ be an integer.
Put $I_r=(x,y^2)^r$ and $R=S/I_r$.
Note that $I_2=(x^2,xy^2,y^4)$.
The ideal $I_r$ of $S$ is minimally generated by the $r$-minors of the $(r+1)\times r$ matrix $\left(\begin{smallmatrix}
y^2&0&\cdots&0\\
x&y^2&\cdots&0\\
0&x&\cdots&0\\
\cdots&\cdots&\cdots&\cdots\\
0&0&\cdots&y^2\\
0&0&\cdots&x
\end{smallmatrix}\right)$.
The Hilbert--Burch theorem \cite[Theorem 1.4.17]{BH} ensures $\r(R)=r$.
It follows from \cite[Corollary 6.5]{burch} that $R$ is Burch.
Using Remark \ref{r24}(5) and Lemma \ref{l217}(1), we have $\Burch(R)\in\{1,2\}$.
Applying Proposition \ref{211}(2), we obtain $\Burch(R)=1$.
As $I_1$ is a complete intersection, it is easy to verify that $\ell(R)=\ell(S/I_r)=\ell(S/I_1)\binom{r+1}{2}=r(r+1)$.
Thus we have $\ell(R)\not<\r(R)(\r(R)+1)$ and $\ell(R)<2\r(R)(\r(R)+1)$.
\end{ex}

The above example can be extended to the following theorem.

\begin{thm} \label{t212}
Let $(R,\m,k)$ be an artinian local ring with $\edim R=2$.
Then the following assertions hold.
\begin{enumerate}[\rm (1)]
\item If there is an equality $\ell(R)< 2\r(R)(\r(R)+1)$, then one has $\Burch(R)\ge 1$, i.e., the local ring $R$ is Burch.
\item If the inequality $\ell(R)< \r(R)(\r(R)+1)$ holds, then there is an equality $\Burch(R)=2$.
\end{enumerate}
\end{thm}

\begin{proof}
Cohen's structure theorem guarantees that $R$ is isomorphic to the quotient $S/I$ of some regular local ring $(S,\n,k)$ of dimension two by an ideal $I\subseteq \n^2$.
Put $r=\r(R)$.
The Hilbert--Burch theorem implies that the $S$-module $S/I$ has a minimal free resolution $0 \to S^{\oplus r} \xrightarrow{\varphi} S^{\oplus(r+1)} \to S \to S/I \to 0$, and that $I=\I_r(\varphi)$.

(1) Remark \ref{r24}(5) says that $R$ is Burch if $\Burch(R)\ge1$.
Now, assume that $R$ is non-Burch.
Proposition \ref{211}(1) says that $\I_1(\varphi)$ is contained in $\n^2$.
Hence $\n I=\n\I_r(\varphi)\subseteq\n\I_1(\varphi)^r\subseteq \n^{2r+1}$, so $\n I\subseteq I\cap\n^{2r+1}$.
We have
$$
\begin{array}{l}
r+1
=\ell(I/\n I)
\ge\ell(I/I\cap\n^{2r+1})
=\ell(I+\n^{2r+1}/\n^{2r+1})
=\ell(\n^{2r}/\n^{2r+1})-\ell(\n^{2r}/I+\n^{2r+1})\\
\phantom{r+1=\ell(I/\n I)}
\ge\ell(\n^{2r}/\n^{2r+1})-\ell(\n^{2r}/I)
=2r+1+\ell(S/\n^{2r})-\ell(S/I)
=2r+1+\binom{2r+1}{2}-\ell(S/I),
\end{array}
$$
and it follows that $\ell(S/I)\ge\frac{2r+1}{2}+r=2r(r+1)$.
Thus the proof of the assertion is completed.

(2) By (1), Remark \ref{r24}(5) and Lemma \ref{l217}(1), we have $\Burch(R)\le2$.
Assume $\Burch(R)\le 1$.
Proposition \ref{211}(2) implies that $\I_1(\varphi)\subseteq (x,y^2)$ for some minimal generators $x,y$ of $\n$.
Then $I=\I_r(\varphi) \subseteq \I_1(\varphi)^r\subseteq (x,y^2)^r$.
Using what we saw in Example \ref{ex34}, we obtain $\ell(S/I) \ge \ell(S/(x,y^2)^r)=r(r+1)$.
\end{proof}

The two statements of the following example say respectively that the assumption of (1) in Theorem \ref{t212} cannot be weakened, and that each of the converses of (1) and (2) in Theorem \ref{t212} can fail.

\begin{ex}
Let $S$ be a regular local ring of dimension two.
Take a regular system of parameters $x,y$ of $S$.
Let $r$ be a positive integer.
\begin{enumerate}[(1)]
\item
Put $I_r=(x^2,y^2)^r$ and $R=S/I_r$.
The ideal $I_r$ of $S$ is minimally generated by the $r$-minors of the $(r+1)\times r$ matrix $\left(\begin{smallmatrix}
y^2&0&\cdots&0\\
x^2&y^2&\cdots&0\\
0&x^2&\cdots&0\\
\cdots&\cdots&\cdots&\cdots\\
0&0&\cdots&y^2\\
0&0&\cdots&x^2
\end{smallmatrix}\right)$. 
By an observation similar to the one in Example \ref{ex34}, we obtain that  $\r(R)=r$ and $\ell(R)=2r(r+1)$.
Also, it follows from \cite[Corollary 6.5]{burch} that $R$ is not Burch.
\item
Put $I_r=(x^r,xy,y^r)$ and $R=S/I_r$.
The ideal $I_r$ of $S$ is minimally generated by the $2$-minors of the $3\times 2$ matrix $\left(\begin{smallmatrix}
y & 0\\
x^{r-1} & y^{r-1}\\
0 & x
\end{smallmatrix}\right)$.
Proposition \ref{211} shows that $R$ is Burch with $\Burch(R)=2$.
By a direct computation, we get $\r(R)=2$ and $\ell(R)=2r-1$.
Hence, for $r\ge 7$, we have $\ell(R)\ge 12=2\r(R)(\r(R)+1)$.
\end{enumerate}
\end{ex}

As an application of the above theorem, we obtain the corollary below, which yields numerical sufficient conditions for a given artinian non-Gorenstein local ring to be a Burch ring.

\begin{cor}\label{18}
\begin{enumerate}[\rm(1)]
\item
Every non-Gorenstein local ring $R$ with $\edim R=2$ and $\ell(R)\le 11$ is a Burch ring.
\item
Let $R$ be non-Gorenstein of $\ell(R)\le 5$.
Then $\Burch(R)=\edim R$.
In particular, $R$ is a Burch ring.
\end{enumerate}
\end{cor}

\begin{proof}
(1) Since $r(R)\ge 2$, we have $\ell(R)\le 11 <12\le 2r(R)(r(R)+1)$.
Theorem \ref{t212}(1) deduces the assertion.

(2) We first deal with the case $\edim R\le 2$.
Since $R$ is non-Gorenstein, we have $\edim R=2$ and $\r(R)\ge 2$.
As $\ell(R)\le5<6=2(2+1)\le\r(R)(\r(R)+1)$, it follows from Theorem \ref{t212}(2) that $\Burch(R)=2=\edim R$.

Now, let $\edim R\ge 3$.
Since $\ell(R)\le 5$, the possible Hilbert functions are $(1,3)$, $(1,3,1)$ and $(1,4)$.
It follows that $R$ is stretched.
Proposition \ref{p210}(1) and Lemma \ref{l217}(2) imply that $\Burch(R)=\edim R$.
\end{proof}

We have slightly more detailed observations by using Hilbert functions in the case of length at most $7$.

\begin{rem} \label{r336}
Let $(R,\m,k)$ be an artinian local ring which satisfies $\edim R\ge3$.
\begin{enumerate}[(1)]
\item
Suppose that $\ell(R)=6$.
Then the possible Hilbert functions are $(1,3,2)$, $(1,3,1,1)$, $(1,4,1)$ and $(1,5)$.
Except the case of $(1,3,2)$, the ring $R$ is stretched.
In this case, $R$ is either Gorenstein or Burch by Proposition \ref{p210}(1).
Examples of a non-Gorenstein non-Burch local ring with type $2$ and Hilbert function $(1,3,2)$ are presented in Example \ref{ex316} and Remark \ref{19}.
\item
Suppose that $\ell(R)=7$.
Then the Hilbert function of $R$ is one of the following.
$$
{\rm(a)}\ (1,3,3),\ \ 
{\rm(b)}\ (1,3,2,1),\ \ 
{\rm(c)}\ (1,3,1,1,1),\ \ 
{\rm(d)}\ (1,4,2),\ \ 
{\rm(e)}\ (1,4,1,1),\ \ 
{\rm(f)}\ (1,5,1),\ \ 
{\rm(g)}\ (1,6).
$$
Let us consider the Gorenstein and Burch properties of $R$ in each of these cases.
\begin{itemize}
\item
Cases (c), (e), (f), (g):
The ring $R$ is stretched, so it is Burch or Gorenstein by Proposition \ref{p210}(1).
\item
Case (a):
As $\soc R\supseteq\m^2$, we get $\r(R)\ge \ell(\m^2)=\edim R=3$.
Proposition \ref{p210}(2) shows $R$ is Burch.
\item
Case (d):
One does not necessarily have $R$ is Gorenstein or Burch.
In fact, the artinian local ring
\[
R=k[x,y,z,w]/(x^2,y^2,xz,xw,yz,yw,z^2,w^2)\cong k[x,y]/(x^2,y^2)\times_k k[z,w]/(z^2,w^2)
\]
has Hilbert function $(1,4,2)$, and is neither Gorenstein nor Burch by \cite[Propositions 5.1 and 6.15]{burch}.
By the way, the local ring $R$ is still uniformly dominant; see Proposition \ref{t220} stated later.
\item
Case (b):
The local ring $R$ can be neither Gorenstein nor Burch; see Example \ref{ex316} stated later.
\end{itemize}
\end{enumerate}
\end{rem}

\section{Exact pairs of zero-divisors}

In this section, as a continuation of the discussion made in the previous section, our main focus is on artinian local rings of length six.
We start by recalling the definition of an exact pair of zero-divisors.

\begin{dfn}
A pair $(x,y)$ of elements of $\m$ is called an \emph{exact pair of zero-divisors} of $R$ if the equalities $0:_Rx=yR$ and $0:_Ry=xR$ hold.
\end{dfn}

The following technical lemma is necessary in the proof of the main result of this section.

\begin{lem} \label{l219}
Let $R$ be an artinian local ring with maximal ideal $\m$ and residue field $k$.
\begin{enumerate}[\rm(1)]
\item
Suppose that $R$ has Hilbert function $(1,n,n-1)$ with $n\ge2$.
Let $x,y$ be elements of $\m\setminus \m^2$.
Then $(x,y)$ is an exact pair of zero-divisors of $R$ if and only if $xy=0$ and $x\m=y\m=\m^2$.
\item
If $\hh_2(R)\le 2$, then there exists an element $v\in \m$ such that $v\m=\m^2$.
\end{enumerate}
\end{lem}

\begin{proof}
(1) The ``if'' part of the equivalence follows from \cite[Lemma 4.3]{CJRSW}.
We show the ``only if'' part.
Assume that $(x,y)$ is an exact pair of zero-divisors.
As $\m^3=0$, we have that $(x)=0:y\supseteq0:\m\supseteq\m^2$.
Therefore, $\m^2=(x)\cap\m^2\subseteq x(\m^2:x)\subseteq x\m$, and the equality $\m^2=x\m$ follows.
Similarly, we get the equality $\m^2=y\m$.

(2) If $\hh_2(R)=0$, then $\m^2=0=0\m$.
Consider the case $\hh_2(R)=1$.
The quotient $\m^2/\m^3$ is a $1$-dimensional $k$-vector space spanned by $\{\overline{xy}\mid x,y\in \m\setminus \m^2\}$.
Since every set of generators of a vector space $V$ contains a basis of $V$, there are elements $x,y\in \m\setminus \m^2$ such that $\overline{xy}$ is a basis of $\m^2/\m^3$.
Nakayama's lemma shows $\m^2=(xy)=x\m$.
Finally, we consider the case $\hh_2(R)=2$.
Then $\m^2/\m^3$ is a $2$-dimensional $k$-vector space spanned by $\{\overline{xy}\mid x,y\in \m\setminus \m^2\}$.
There are elements $x,y,z,w\in \m\setminus \m^2$ such that $\overline{xy}, \overline{zw}$ is a basis of $\m^2/\m^3$.
If $\overline{xw}, \overline{xy}\in\m^2/\m^3$ is a basis of $\m^2/\m^3$, then $\m^2=(xw,xy)=x\m$ by Nakayama's lemma.
If $\overline{xw}=u\overline{xy}$ for some $u\in k^\times$, then $\overline{xw},\overline{zw}$ is a basis of $\m^2/\m^3$, so that $\m^2=(xw,zw)=w\m$.
Thus we may assume $\overline{xw}=0$.
Similarly, we may assume $\overline{zy}=0$.
Then $\overline{(y+w)x},\overline{(y+w)z}$ is a basis of $\m^2/\m^3$, and we get $\m^2=(y+w)\m$.
\end{proof}

The following theorem is the main result of this section.
It describes the structure of a non-Gorenstein non-Burch local ring of length six.

\begin{thm} \label{p311}
Let $(S,\n,k)$ be a regular local ring, and let $I$ be an $\n$-primary ideal of $S$ such that $I\subseteq \n^2$ and $\ell(S/I)=6$.
Assume that the local ring $S/I$ is neither Gorenstein nor Burch.
Then one has $\edim S=3$, $\r(S/I)=2$, $\nu(I)=4$, and there exists an exact pair of zero-divisors in $S/I$.
\end{thm}

\begin{proof}
Put $R=S/I$ and $\m=\n/I$.
It is seen from Corollary \ref{18}(1) that $\edim R\ge3$.
By Remark \ref{r336}(1), the Hilbert function of $R$ must be $(1,3,2)$, and thus $\edim S=\edim R=3$.
It follows from Proposition \ref{p210}(1) that $\m^2=\soc R$, which implies $I:\n=\n^2$ and $\r(R)=2$.
Hence $\nu(I:\n)=\nu(\n^2)=6$.
As $R$ is not Burch, we have $\n I=\n(I:\n)=\n^3$.
Since $R$ is not a complete intersection, we have $\height I\ne\nu(I)$.
Applying Lemma \ref{l25}, we get $3=\height I<\nu(I)\le\nu(I:\n)-\r(R)=6-2=4$, and thus $\nu(I)=4$.
Now, the only thing we have to show that $R$ admits an exact pair of zero-divisors.

Using Lemma \ref{l219}(2), we can choose an element $x\in \m$ such that $x\m=\m^2$.
Nakayama's lemma ensures $x\notin\m^2$.
We observe that $\ell((x))=1+\ell(x\m)=1+2=3$.
The exact sequence $0 \to 0:x\to R \to (x) \to 0$ shows that $\ell(0:x)=\ell(R)-\ell((x))=3$, which implies that $0:x\not\subseteq \m^2$.
Take $y\in(0:x)\setminus \m^2$.
In view of Lemma \ref{l219}(1), we may assume 
$y\m\not=\m^2$, which yields $\ell((y))=1+\ell(y\m)<1+\ell(\m^2)=3$.
Also, $y$ is not in $\m^2=\soc R$, i.e., $\ell((y))>1$.
It follows that $\ell((y))=2$.
We also observe that $\ell(0:y)=\ell(R)-\ell((y))=4$ and $\ell((0:y)/\m^2)=\ell(0:y)-\ell(\m^2)=2$.
Let $\pi\colon S \to R$ be the canonical surjection, and $X,Y$ part of a minimal system of generators of $\n$ such that $\pi(X)=x$ and $\pi(Y)=y$.
We proceed by setting several cases.
\begin{enumerate}[(1)]
\item
We consider the case $y^2=0$.
Write $\m=(x,y,z)$ and $\n=(X,Y,Z)$ with $\pi(Z)=z$.
We have $\m^2=x\m =x(x,y,z)=(x^2,xz)$.
In other words, $x^2,xz$ form a $k$-basis of $\m^2=\m^2/\m^3$.
Note that $0\not=y\m=y(x,y,z)=(yz)$.
There exist $a,b,c,d\in R$ such that $yz=ax^2+bxz$,  $z^2=cx^2+dxz$.
For any lifts $A,B,C,D\in S$ of $a,b,c,d$, we see that
\[
J:=(XY,Y^2,YZ-AX^2-BXZ,Z^2-CX^2-DXZ)\subseteq I.
\]
Note that the residue classes in $\n^2/\n^3$ of the quadratics appearing above are linearly independent over $k$.
Since $J\subseteq I\subseteq \n^2$, the map $J \to\n^2/\n^3$ factors through $I/\n I$.
Therefore, the residue classes in $I/\n I$ of those quadratics are also linearly independent.
Thus, they are part of a minimal system of generators of $I$.
As $\dim_k(I/\n I)=\ell(I)=4$, they actually form a minimal system of generators of $I$, that is, $J=I$.
\begin{enumerate}
\item[(1a)]
Assume $a\notin \m$.
Then $A$ is a unit of $S$.
Replacing $Y$ with $A^{-1}Y$, we may assume $A=1$.
We have
$$
I+(Y)=(Y,X^2+BXZ,Z^2-CX^2-DXZ)=(Y,X(X+BZ), Z((BC-D)X+Z)).
$$
Suppose that $1-B(BC-D)\in \n$.
Then, modulo $(X+BZ)+\n^2$ we have
$$
(BC-D)X+Z\equiv(BC-D)(-BZ)+Z=Z(1-B(BC-D))\equiv0,
$$
and hence $I+(Y) \subseteq (Y,X+BZ)+\n^3 \subseteq (Y,X+BZ)+\n(I+(Y))$ as $\n I=\n^3$.
Nakayama's lemma shows that $I+(Y) \subseteq (Y,X+BZ)$, which contradicts the fact that $I$ is $\n$-primary.
It turns out that $1-B(BC-D)\not\in \n$, so that $1-b(bc-d)\not\in \m$.
Set $w=(d-bc)x+cy-z$.
As $xz,yz=x^2+bxz$ form a $k$-basis of $\m^2$, we have $\m^2=(xz,yz)=z\m$.
Also, we have $z^2=cx^2+dxz=cyz+(d-bc)xz$, which gives $zw=0$.
The matrix $\left(\begin{smallmatrix}d-bc&-1\\-1&-b\end{smallmatrix}\right)$ is invertible as $b(bc-d)-1\not\in \m$. 
Hence $xw=(d-bc)x^2-xz$ and $yw=-yz=-x^2-bxz$ form a $k$-basis of $\m^2$ since so do $x^2$ and $xz$.
Thus $\m^2=(xw,yw)=w\m$.
Lemma \ref{l219}(1) implies $(z,w)$ is an exact pair of zero-divisors in $R$.
\item[(1b)]
Assume $a\in \m$.
Putting $z'=z+x$, we have $\m=(x,y,z')$, $yz'=(a-b)x^2+bxz'$ and ${z'}^2=(c-d-1)x^2+(d+2)xz'$.
As $0\ne yz=ax^2+bxz$ and $a\in\m$, we must have $b\notin\m$, and hence $a-b\notin\m$.
The argument (1a) guarantees the existence of exact pairs of zero-divisors in $R$.
\end{enumerate}
\item
We consider the case $y^2\not=0$.
Since $(0:y)/\m^2$ has length $2$ and does not contain $\overline y$, we have $\m/\m^2=(0:y)/\m^2\oplus k\overline y$.
Hence, there exists $z\in \m\setminus \m^2$ such that $yz=0$ and that $\{\overline x,\overline y,\overline z\}$ is a $k$-basis of $\m/\m^2$.
Then $\m=(x,y,z)$.
Write $\n=(X,Y,Z)$ with $\pi(Z)=z$.
Since $\m^2=x\m=(x^2,xz)$, there exist $a,b,c,d\in R$ with $y^2=ax^2+bxz$,  $z^2=cx^2+dxz$.
For any lifts $A,B,C,D\in S$ of $a,b,c,d\in R$, we have $J:=(XY,YZ,Y^2-AX^2-BXZ,Z^2-CX^2-DXZ)\subseteq I$.
The same argument as in (1) shows $J=I$.
\begin{enumerate}
\item[(2a)]
Assume $a\notin\m$.
Then $A$ is a unit of $S$.
We have
$$
\begin{array}{l}
I+(Y)
=(Y,AX^2+BXZ,Z^2-CX^2-DXZ)\\
\phantom{I+(Y)}=(Y,X(AX+BZ), Z(Z-(D-A^{-1}BC)X)).
\end{array}
$$
Similarly as in (1a), if $A^2-B(BC-AD)\in\n$, then
$$
Z-(D-A^{-1}BC)X
\equiv Z-(D-A^{-1}BC)(-A^{-1}BZ)
=A^{-2}Z(A^2-B(BC-AD))
\equiv0
$$
modulo $(AX+BZ)+\n^2$, and we can derive a contradiction.
We thus obtain $A^2-B(BC-AD)\notin \n$, which shows $a^2-b(bc-ad)\not\in \m$.
Set $w=ax-y+bz$.
Observe that $(x+y)x=x^2$ and $(x+y)z=xz$ form a $k$-basis of $\m^2$, and thus $(x+y)\m=\m^2$.
Also $y(x+y)=y^2=ax^2+bxz=ax(x+y)+bz(x+y)$, which yields that $(x+y)w=0$.
Finally, the condition $a(bd+a)-b^2c=a^a-b(bc-ad)\notin \m$ ensures that the elements $xw=ax^2+bxz \text{ and } zw=bcx^2+(bd+a)xz$ form a $k$-basis of $\m^2$.
Thus, $w\m=\m^2$.
Lemma \ref{l219}(1) implies that $(x+y,w)$ is an exact pair of zero-divisors in $R$.
\item[(2b)]
Assume $a\in \m$.
Putting $z'=z+x$, we have $\m=(x,y,z')$, $yz'=0$, $y^2=(a-b)x^2+bxz'$ and ${z'}^2=(c-d-1)x^2+(d+2)xz'$.
As $0\ne y^2=ax^2+bxz$ and $a\in\m$, we must have $b\notin\m$, and hence $a-b\notin\m$.
The argument (2a) guarantees the existence of exact pairs of zero-divisors in $R$.
\end{enumerate}
\end{enumerate}
Now the proof of the theorem is completed.
\end{proof}

Next, we review the definitions of Tor/Ext-friendliness, embedding deformations and G-regularity, and then remark how these properties of a local ring relate to dominance.

\begin{dfn}
\begin{enumerate}[(1)]
\item
We say that a local ring $R$ is {\em Tor-friendly} provided that if $M$ and $N$ are $R$-modules such that $\Tor^R_i(M,N)=0$ for all $i\gg0$, then one has either $\pd_RM<\infty$ or $\pd_RN<\infty$.
\item
Dually to (1), we say that $R$ is {\em Ext-friendly} provided that if $M,N$ are $R$-modules such that $\Ext_R^i(M,N)=0$ for all $i\gg0$, then one has either $\pd_RM<\infty$ or $\id_RN<\infty$.
\item
We say that $R$ has an \emph{embedded deformation} if there exist a local ring $S$ with maximal ideal $\n$ and an $S$-regular element $f\in \n^2$ such that $\widehat{R}\cong S/(f)$.
\item
Following \cite{greg}, we say that a local ring $R$ is \emph{G-regular} if every totally reflexive $R$-module is free.
\end{enumerate}
\end{dfn}

\begin{rem} \label{r312}
\begin{enumerate}[(1)]
\item
If $R$ is Tor-friendly, then it is Ext-friendly by \cite[Proposition 5.5]{AINS}.
If $R$ is non-Gorenstien and Ext-friendly, then it is G-regular by definition.
\item
If $(x,y)$ is an exact pair of zero-divisors of $R$, then $R/(x)$ and $R/(y)$ are nonfree totally reflexive $R$-modules by \cite[Proposition 2.4]{greg}, and hence $R$ is not G-regular.
\item 
If $R$ has an embedded deformation, then $R$ is not G-regular.
This follows from \cite[Proposition 4.6 and Corollary 4.7]{greg} (note that it is assumed in \cite[Proposition 4.6]{greg} that the local ring is G-regular, but this assumption is not used in the proof of its ``only if'' part).
Also, if $R$ has an embedded deformation and is Ext-friendly, then it is a hypersurface.
This follows from \cite[Corollary 4.4]{Se} (note that it is assumed in \cite[Corollary 4.4]{Se} that the local ring is Cohen--Macaulay, but this assumption is not used in its proof).
\item
Suppose that $R$ is Cohen--Macaulay of codimension at most three.
Then $R$ is Tor-friendly if and only if $R$ is Ext-friendly, if and only if $R$ has no embedded deformations; see \cite[Theorems 3.2 and 3.8]{LM}.
\item
By (1), (2), (3), \cite[Theorem 1.1]{dlr} and \cite[Corollary 5.5]{udim} one has the following hierarchy:
\[
\xymatrix@R-1pc@C-2.1pc{
\qquad\text{Burch}\ar@{=>}[rrrr]&&&&
\text{uniformly dominant}\ar@{=>}[d]&&
\text{Ext-friendly}\ar@{=>}[rr]_-{\rm(a)} \ar@{=>}[rrd]_-{\rm(b)} &&
\text{G-regular}\ar@{=>}[r]\ar@{=>}[d]&
\text{no exact pair of zero-divisors}\\
&&&&\text{dominant}\ar@{=>}[rr]&&
\text{Tor-friendly}\ar@{=>}[u]&&
\text{no embedded deformation}}
\]
where it is assumed that the ring is not Gorenstein for (a) and not a hypersurface for (b).
\item
An artinian local ring $R$ whose Hilbert function has the form $(1,e,e-1)$ does not necessarily admit an exact pair of zero-divisors.
Indeed, let $S=k[\![x,y,z]\!]$ be the formal power series ring over a field $k$, and $\n=(x,y,z)$ the maximal ideal of $S$.
Let $I=(x^2,xy,xz,y^2,yz^2,z^3)$ be the ideal of $S$, and put $R=S/I$.
Then $R$ is an artinian local $k$-algebra with Hilbert function $(1,3,2)$.
Since $\n I:\n\not\ni z^2\in I:\n$, the ring $R$ is Burch.
It follows from (5) that $R$ does not have an exact pair of zero-divisors.
\item
In relation to (6), for an infinite field $k$, every ``generic'' standard graded $k$-algebra whose Hilbert function is of the form $(1,e,e-1)$ has an exact pair of zero-divisors; see \cite[Corollary 8.5]{CJRSW} for details.
We also refer the reader to \cite{HeS,KSV} for further investigation on exact pairs of zero-divisors, and to \cite{Y1} for other approaches to artinian local rings whose Hilbert functions have the form $(1,e,e-1)$.
\end{enumerate}
\end{rem}

In view of the diagram of implications given in the above Remark \ref{r312}(5), it is natural to ask the following.

\begin{ques}
Does there exist an implication between having no embedded deformation and having no exact pair of zero-divisors?
\end{ques}

As a consequence of the theorem stated above, we obtain the following result, which says that in the case where the base local ring has low length, various conditions discussed so far turn out to be equivalent.

\begin{cor}\label{20}
For an artinian non-Gorenstein local ring $R$ with $\ell(R)\le 6$ the following are equivalent.\\
{\rm(1)} $R$ is Burch.
{\rm(2)} $R$ is uniformly dominant.
{\rm(3)} $R$ is dominant.
{\rm(4)} $R$ is Tor-friendly.
{\rm(5)} $R$ is Ext-friendly.\\
{\rm(6)} $R$ is G-regular.
{\rm(7)} $R$ has no embedded deformation.
{\rm(8)} $R$ has no exact pair of zero-divisors.
\end{cor}

\begin{proof}
Theorem \ref{p311} and Corollary \ref{18}(2) show (8) implies (1).
Theorem \ref{p311} says $R$ is Burch unless $\edim R=3$.
By Remark \ref{r312}(4) we have (7) implies (4).
The remaining implications are shown by Remark \ref{r312}(5).
\end{proof}

\begin{rem}\label{36}
\begin{enumerate}[(1)]
\item
The implication (3)$\Rightarrow$(1) in Corollary \ref{20} is no longer true if we remove the assumption $\ell(R)\le6$.
In fact, there is a dominant local ring $R$ with $\ell(R)=7$ which is not Burch; see Remark \ref{r336}(2).
\item
The implication (6)$\Rightarrow$(5) in Corollary \ref{20} is no longer true if we remove the assumption that $\ell(R)\le6$.
Indeed, let $k$ be a field, and set $S=k[\![x,y,z,w]\!]$, $R=S/(x^2,xy,y^2,z^2,zw,w^2)$, $T=S/(x^2,xy,y^2)$ and $C=\Ext_T^2(R,T)$.
Then $R$ is an artinian non-Gorenstein local ring with $\ell(R)=9$, and $C$ is a semidualizing $R$-module with $\pd_RC=\id_RC=\infty$; see \cite[Example 1.2]{JLS}.
Hence $R$ is not Ext-friendly.
On the other hand, the maximal ideal $\m$ of $R$ is such that $\m^3=0$, and $R$ has Hilbert function $(1,4,4)$.
It follows from \cite[Theorem 3.1]{Y1} that $R$ is G-regular.
Thus, (5) is not satisfied but (6) is.
\end{enumerate}
\end{rem}

The ring $R$ in Remark \ref{36}(2) is such that $\ell(R)=9$, so it is natural to ask whether the implication (6)$\Rightarrow$(5) in Corollary \ref{20} holds or not when $\ell(R)=7,8$.
We do not know the answer, so we pose the following question.

\begin{ques}
Let $R$ be a G-regular non-Gorenstein local ring with length $7$ or $8$.
Is then $R$ Ext-friently?
\end{ques}

Now we shall construct an explicit example of an artinian non-Gorenstein non-Burch local ring which does possess an exact pair of zero-divisors.
Note that such a ring can have length $6$ and Hilbert function $(1,3,2)$ or length $7$ and Hilbert function $(1,3,2,1)$.

\begin{ex} \label{ex316}
Let $S=k[\![x,y,z]\!]$ be a formal power series ring over a field $k$, and denote by $\n$ the maximal ideal of $S$.
Fix an integer $n\ge 2$.
Let $I=(xy,xz,yz,x^2+y^2+z^n)$ be an ideal of $S$, and put $R=S/I$.
The ideal $J=(x^3,xy,y^3,xz,yz,z^{n+1})$ of $S$ is contained in $I$.
Hence the Hilbert function of $R$ satisfies
\[
\hh_i(R) \le \hh_i(S/J)=1 \text{ for all $3\le i \le n$,}\quad \text{and}\quad\hh_i(R)\le \hh_i(S/J)=0 \text{ for all $i\ge n+1$.}
\]
The containment $I\subseteq \n^2$ shows that $\hh_1(R)=3$.
Observing that $I+\n^3$ is equal to $(xy,xz,yz,x^2+y^2,z^3)$ for $n\ge 3$ and to $(xy,xz,yz,x^2+y^2-z^2)$ for $n=2$, we deduce that $\hh_2(R)=2$.
Also, since $I+(x,y)=(x,y,z^n)$, we have that $\hh_i(R) \ge \hh_i(S/I+(x,y))=1$ for $3\le i\le n$.
In summary, the following hold true.
\[
\hh_1(R)=3, \quad \hh_2(R)=2, \quad \hh_i(R) = 1 \text{ for all $3\le i \le n$,} \quad\text{and}\quad\hh_i(R)=0 \text{ for all $i\ge n+1$.}
\]
Thus we have $\ell(R)=n+4$.
It is worth writing down some special cases here:
the artinian local ring $R$ has length $6$ and Hilbert function $(1,3,2)$ for $n=2$, and has length $7$ and Hilbert function $(1,3,2,1)$ for $n=3$.

We claim that  $(x+y-z,x+y-z^{n-1})$ is an exact pair of zero-divisors of $R$.
Indeed, the containment $(x+y-z) \subseteq \ann_R(x+y-z^{n-1})$ is immediate.
As $R/(x+y-z) \cong k[x,y]/(xy,x^2,y^2)$ and $R/(x+y-z^{n-1}) \cong k[x,z]/(xz,x^2,z^n)$, we have $\ell(R/(x+y-z))=3$ and $\ell(R/(x+y-z^{n-1}))=n+1$.
Thus, we get 
\[
\ell(R/(x+y-z^{n-1}))=n+1=\ell(R)-\ell(R/(x+y-z))=\ell((x+y-z)).
\]
Using the exact sequence $0 \to \ann(x+y-z^{n-1}) \to R \to R \to R/(x+y-z^{n-1}) \to 0$, we see that the equalities $\ell(\ann(x+y-z^{n-1}))=\ell(R/(x+y-z^{n-1}))=\ell((x+y-z))$ hold.
The equality $(x+y-z)=\ann_R(x+y-z^{n-1})$ now follows.
We conclude that $(x+y-z,x+y-z^{n-1})$ is an exact pair of zero-divisors.
We also note that $R$ is not Burch by Remark \ref{r312}(5), and that $R$ is not Gorenstein because its socle $(x^2,z^n)$ has length two.
\end{ex}

We close the section by remarking that applying the main result of this section, we can determine the structure of a non-Gorenstein non-Burch local algebra over an algebraically closed field with length six.

\begin{rem}\label{19}
Let $k$ be an algebraically closed field.
The isomorphism classes of commutative artinian local $k$-algebras of length at most six are completely classified in \cite[Table {1}]{P}.
Let $R$ be one of those $k$-algebras.
Suppose that $R$ is neither Gorenstein nor Burch.
Then Corollary \ref{18}(2), Remark \ref{r336}(1) and Theorem \ref{p311} yield that $R=k[x,y,z]/I$, $\r(R)=2$, $\nu(I)=4$, and the Hilbert function of $R$ is $(1,3,2)$.
Thus $I$ is one of the following ideals, where the last two ones are possible only when $k$ has characteristic two.
$$
\begin{array}{l}
(xy,z^2,xz-yz,x^2+y^2-xz),\qquad
(xy,xz,yz,x^2+y^2-z^2),\qquad
(x^2,xy,yz,xz+y^2-z^2),\\
(x^2,xy,y^2,z^2),\qquad
(x^2,z^2,y^2-xz,yz),\qquad
(x^2,xy,y^2,z^2-xz).
\end{array}
$$
Conversely, in the case where $I$ is one of these ideals, the $k$-algebra $R$ is an artinian non-Gorenstein non-Burch local ring with $\ell(R)=6$, $\r(R)=2$, $\nu(I)=4$ and Hilbert function $(1,3,2)$.
\end{rem}

\section{Upper bounds for dominant indices}

In this section we explore upper bounds for the dominant index in various kinds of situations.
The following lemma is helpful to compute (a bound of) the dominant index through modding out by a regular sequence.
Assertion (2b) of the lemma can be shown by an elementary proof, which is a special case of \cite[Theorem B]{LS}.
Assertion (3) of the lemma will be used frequently, the first item of which is taken from \cite{udim}.

\begin{lem} \label{227}
Let $R$ be a local ring with maximal ideal $\m$ and residue field $k$.
\begin{enumerate}[\rm(1)]
\item
\begin{enumerate}[\rm(a)]
\item
Suppose that $\depth R=0$ and $\Burch(R)\ge 2$.
Then $k$ is a direct summand of $\syz^i M$ for every nonfree $R$-module $M$ and for every integer $i\ge 5$.
\item
Suppose that $R$ is a nontrivial fiber product.
Then $\m$ is a direct summand of $\syz^3 M\oplus \syz^4 M$ for every $R$-module $M$ of infinite projective dimension.
\end{enumerate}
\item
\begin{enumerate}[\rm(a)]
\item
Suppose that $R=S/I$ where $(S,\n)$ is a regular local ring and $I$ a Burch ideal of $S$.
Let $M$ be a nonfree $R$-module.
Then there exists an element $x\in R$ such that $k$ belongs to $\langle\k(x,M) \rangle_1^{\ds(R)}$.
\item
Let $M$ be an $R$-module.
Let $\xx=x_1,\dots,x_n$ and $\yy=y_1,\dots,y_m$ be sequences of elements of $R$, and assume that $\xx$ is regular on $R,M$.
Let $C=\k(\yy,M/\xx M)$ be the Koszul complex.
Then the inequality $\level^M_{\ds(R)}(k) \le (n+m+1)(\level_{\ds(R/(\xx))}^C(k)+1)-1$ holds.
\item Let $I$ be an ideal of $R$ with $\Sing R\subseteq\V(I)$ (e.g., $I=\bigcap_{\p\in \Sing R}\p$) and put $h=\dim R/I$.
Then for each nonzero object $X$ of $\ds(R)$ there exists an $R$-module $M$ of infinite projective dimension such that $M_\p$ is $R_\p$-free for all $\p\in\spec R\setminus\{\m\}$, and $M$ belongs to $\langle X\rangle_{h+1}^{\ds(R)}$.
\end{enumerate}
\item
\begin{enumerate}[\rm(a)]
\item
Let $x\in\m$ be an $R$-regular element.
Then the inequality $\dx(R)\le2\dx(R/(x))+1$ holds.
If $x$ is not in $\m^2$, then the inequality $\dx(R/(x))\le2\dx(R)+1$ holds as well.
\item
Let $\xx=x_1,\dots,x_n$ be an $R$-sequence.
Then there is an inequality $\dx(R)\le(n+1)(\dx(R/(\xx))+1)-1$.
\end{enumerate}
\end{enumerate}
\end{lem}

\begin{proof}
(1a) The assertion follows from \cite[Theorem A]{DM}; see also \cite[Theorem 4.1]{DE}.

(1b) The assertion follows from \cite[Theorem 2.9]{udim}; see also \cite[Theorem A]{fiber}.

(2a) Since $I$ is a Burch ideal, we find an element $z\in \n \setminus (\n I: (I:\n))$.
Let $x$ be the image of $z$ in $R$.
There is an exact sequence $0\to\syz M\to F\to M\to0$ with $F$ free.
The proof of \cite[Lemma 7.2(1)]{burch} yields an exact sequence $0 \to \syz M \to N \to M \to 0$ with $x\in\I_1(N)$ and an isomorphism $(0\to N\to F\to0)\cong\k(x,M)$ in $\db(\mod R)$.
There is an exact triangle $F \to\k(x,M) \to N[1] \rightsquigarrow$ in $\db(\mod R)$.
Passing to the singularity category, we get an isomorphism $K(x,M)\cong N[1]$ in $\ds(R)$.
Since $x$ is in $\I_1(N)$, it follows from \cite[Proposition 4.4]{DE} that $k$ is a direct summand of $\syz^2 N$.
Thus $k$ belongs to $\langle K(x,M) \rangle_1^{\ds(R)}$.

(2b) We may assume that $l:=\level_{\ds(R/(\xx))}^C(k)$ is finite, and then $k$ belongs to $\langle C\rangle_{l+1}^{\ds(R/(\xx))}$.
The exact functor $-\lten_{R/(\xx)}(R/(\xx))_R:\db(\mod R/(\xx))\to\db(\mod R)$ induces an exact functor $\ds(R/(\xx))\to\ds(R)$ as $(R/(\xx))_R$ (the right $R$-module $R/(\xx)$) has finite projective dimension.
Applying this functor shows that $k$ is in $\langle C\rangle_{l+1}^{\ds(R)}$.
Since $\xx$ is regular on $M$, the $R$-module $M/\xx M$ is quasi-isomorphic to the Koszul complex $\k(\xx, M)$.
There are isomorphisms $C\cong M/\xx M\lten_R\k(\yy,R) \cong\k(\xx,M)\lten_R\k(\yy,R) \cong\k(\xx,\yy,M)$ in $\db(\mod R)$.
The Koszul complex $\k(\xx,\yy,M)$ induces a series of exact triangles $\{M^{\oplus\binom{n+m}{i}}[i] \to K_i \to K_{i+1} \rightsquigarrow\}_{0\le i\le n+m}$, where $K_0=\k(\xx,\yy,M)$ and $K_{n+m+1}=0$.
By induction we get $\k(\xx,\yy,M)\in \langle M\rangle_{n+m+1}^{\db(\mod R)}$.
As $K(\xx,\yy,M)\cong C$, it follows that $k$ is in $\langle M\rangle_{(n+m+1)(l+1)}^{\ds(R)}$, which yields the desired inequality $\level^M_{\ds(R)}(k)\le (n+m+1)(l+1)-1$.

(2c) Let $\xx=x_1,\dots,x_h$ be a sequence of elements of $R$ such that the image of $\xx$ in $R/I$ forms a system of parameters of $R/I$.
Then $I+(\xx)$ is an $\m$-primary ideal of $R$.
Decomposing the Koszul complex $\k(\xx,R)$ into exact triangles and applying the exact functor $X\lten_R-$, we have that $Y:=X\lten_R\k(\xx,R)\in\langle X\rangle_{h+1}^{\ds(R)}$.
Put $d=\dim R$.
By Remark \ref{21} we find an $R$-module $N$ and an integer $n$ such that $Y\cong(\syz^dN)[n]$ in $\ds(R)$.
Setting $M=\syz^dN$, we see that $Y\cong M[n]$ in $\ds(R)$ and $M\in\langle X \rangle_{h+1}^{\ds(R)}$.
There are isomorphisms
$$
\textstyle Y\lten_Rk\cong X\lten_R\k(\xx,k)\cong X\lten_R\big(\bigoplus_{i=0}^hk^{\oplus\binom{h}{i}}[i]\big)\cong\bigoplus_{i=0}^h(X\lten_Rk)^{\oplus\binom{h}{i}}[i]
$$
in $\db(\mod R)$, and the last term contains $X\lten_Rk$ as a direct summand.
Hence $Y$ is nonzero in $\ds(R)$, so $\pd_RM=\infty$.
Let $\p$ be a prime ideal of $R$ different from $\m$.
If $\p$ is not in $\V(I)$, then $R_\p$ is regular, and $M_\p$ has finite projective dimension as an $R_\p$-module.
Suppose that $\p$ is in $\V(I)$.
Then $\p$ does not contain $\xx$, and there is an isomorphism $\k(\xx,R)_\p\cong 0$ in $\db(\mod R_\p)$.
It follows that $Y_\p$ is zero in $\ds(R_\p)$, and so is $M_\p$.
Again, $M_\p$ has finite projective dimension as an $R_\p$-module.
As $M=\syz^dN$, we observe that $M_\p$ is $R_\p$-free.

(3) The first assertion is none other than \cite[Theorem 6.2]{udim}.
To prove the second assertion, it suffices to verify that if $r$ is an integer with $\dx(R/(\xx))\le r$, then $\dx(R)\le(n+1)(r+1)-1$.
Let $M$ be a nonzero object of $\ds(R)$.
In view of Remark \ref{21}, we may assume that $M$ is an $R$-module such that $\xx$ is an $M$-sequence.
Applying (2b), we get the inequality $\level^M_{\ds(R)}(k) \le (n+1)(\level^{M/\xx M}_{\ds(R/(\xx))}(k)+1)-1 \le (n+1)(r+1)-1$.
Consequently, the desired inequality $\dx(R)\le(n+1)(r+1)-1$ follows.
\end{proof}

We give a name to the condition of a local ring deforming to a nontrivial fiber product.

\begin{dfn}
Let $(R,\m,k)$ be local.
For each $n\ge0$ we say that $R$ is a {\em nontrivial $n$-fiber product} over $k$ if there is an $R$-sequence $\xx=x_1,\dots,x_n$ such that $R/(\xx)$ is a nontrivial fiber product over $k$.
If there is an integer $n\ge0$ such that $R$ is a nontrivial $n$-fiber product over $k$, then $R$ is called a {\em quasi-fiber product ring}.
In other words, $R$ is a quasi-fiber product ring if and only if $\m$ is {\em quasi-decomposable} in the sense of \cite{fiber}.
\end{dfn}

The following proposition gives upper bounds for dominant indices under certain conditions.

\begin{prop}\label{t220}
Put $e=\edim R$, $r=\r(R)$, $u=\depth R$, $I=\bigcap_{\p\in \Sing R}\p$, and $h=\dim R/I$.
\begin{enumerate}[\rm(1)]
\item
If one of the following holds, then $R$ is uniformly dominant with $\dx(R)\le0$.\\
{\rm(a)}
$R$ is a nontrivial fiber product over $k$.\quad
{\rm(b)}
$\soc R\nsubseteq\m^2$.\\
{\rm(c)}
$R$ is non-Gorenstein, artinian, and stretched.\quad
{\rm(d)}
$\m^2=0$.\quad
{\rm(e)}
$u=0$ and $\Burch(R)\ge 2$.\\
{\rm(f)}
$R\cong S/\n J$, where $(S,\n)$ is a regular local ring with $\dim S\ge 2$ and $J$ is an ideal of $S$.\\
{\rm(g)}
$R\cong S/\n^s$, where $(S,\n)$ is a regular local ring with $\dim S\ge 2$ and $s$ is a positive integer.\\
{\rm(h)}
$R\cong S/(J:\n)$, where $(S,\n)$ is a regular local ring with $\dim S\ge 2$ and $J\subseteq\n^2$ is an ideal of $S$ such that $S/J$ is an artinian Gorenstein ring.\quad
{\rm(i)}
$\m^3=0$ and $\P_k(t)\ne\frac{1}{1-et+rt^2}$.
\item
Suppose that $\syz^{u+1}k$ belongs to $\add(R\oplus \syz^{u+2}k)$.
Then $R$ is uniformly dominant, and the following hold.\\
{\rm(a)}
One has $\dx(R)\le2$ if $u=0$.\quad
{\rm(b)}
One has $\dx(R)\le(h+1)(2u+3)-1$ if $u>0$.\\
{\rm(c)}
One has $\dx(R)\le2u+2$ if $R$ has an isolated singularity.
\item
Suppose that $\syz^{u}k$ belongs to $\add(R\oplus \syz^{u+2}k)$.
Then $R$ is uniformly dominant, and the following hold.\\
{\rm(a)}
One has $\dx(R)\le 1$ if $u=0$.\quad
{\rm(b)}
One has $\dx(R)\le (h+1)(2u+4)-1$ if $u>0$.\\
{\rm(c)}
One has $\dx(R)\le 2u+3$ if $R$ has an isolated singularity.
\item
Suppose that $R$ is Burch.
Then $R$ is uniformly dominant, and the following hold.\\
{\rm(a)}
One has $\dx(R)\le1$ if $u=0$.\quad
{\rm(b)}
One has $\dx(\widehat R)\le u+1$ and $\dx(R)\le(h+1)(2u+4)-1$.
\item
Let $n\ge0$.
Let $R$ be a nontrivial $n$-fiber product over $k$.
Then $R$ is uniformly dominant with $\dx(R)\le n$.
\end{enumerate}
\end{prop}

\begin{proof}
(1) Note by Proposition \ref{p210}(1) that (b) and (c) are special cases of (a), while (d) and (g) are special cases of (b) and (f), respectively.
In view of Remark \ref{r24}(2), we may also assume $\edim R\ge2$.
In case (f), we may assume $0\ne J\subseteq\n$.
Since $\soc S/\n J=(\n J:\n)/\n J$ contains $J/\n J$, which is nonzero, we have $\depth R=0$.
By Lemma \ref{l217}(2a) the assertion in case (f) will follow once it is shown in case (e).
In case (h), we may assume $J:\n\ne\n$, and then the assertion will follow once it is shown in case (e).
Thus, it suffices to prove the assertion in cases (a), (e) and (i) only.

(a) Let $X$ be a nonzero object of $\ds(R)$.
Take an integer $n$ and an $R$-module $M$ such that $\pd M=\infty$ and $X\cong M[n]$.
Thanks to Lemma \ref{227}(1b), the maximal ideal $\m$ is a direct summand of $\syz^3 M\oplus \syz^4 M$.
Therefore, $k[-1]\cong\m \in\langle M[-3]\oplus M[-4]\rangle_{1}$ in $\ds(R)$.
It is seen that $k\in \langle X\rangle_{1}$.
We achieve the inequality $\dx(R)\le 0$.

(e) By using Lemma \ref{227}(1a) instead of Lemma \ref{227}(1b), the assertion is shown similarly to (a).

(i) The assumption $\m^3=0$ implies that $\soc R$ contains $\m^2$.
In view of (b), we may assume $\soc R=\m^2$.
Let $M$ be a nonfree $R$-module.
Then $N:=\syz M\ne0$ is contained in a direct sum of copies of $\m$.
Hence $N$ is killed by $\m^2$.
Since $\P_k(t)\ne\frac{1}{1-et+rt^2}$, we see from \cite[Corollary 3.4(2) and Lemma 3.6]{L} that $k$ is a direct summand of $\syz^nN$ for some $n>0$.
Therefore, $k$ belongs to $\langle M\rangle_1$ and we obtain the inequality $\dx(R)\le0$.

(2) Assertions (a) and (c) follow from \cite[Corollary 5.5(1)]{udim} and assertion (b), respectively.
To show (b), let $X$ be a nonzero object of $\ds(R)$.
Lemma \ref{227}(2c) yields an $R$-module $L$ with $\pd_RL=\infty$, $L\in \langle X\rangle_{h+1}^{\ds(R)}$, and such that $L_\p$ is $R_\p$-free for each $\p\in\spec R\setminus\{\m\}$.
The proof of \cite[Theorem 4.7]{udim} shows that $\syz^{u+1}k\in [\syz^{u+3}L]_{2u+3}$.
Therefore, $\syz^{u+1}k$ is in $\langle X\rangle_{(h+1)(2u+3)}$, and the inequality $\dx(R)\le (h+1)(2u+3)-1$ follows.

(3) Again, (c) follows from (b), while (b) is shown by a similar argument as in the proof of (2).
To show (a), suppose $u=0$.
We may assume that $R$ is singular.
By assumption, the $R$-module $k$ belongs to $\add(R\oplus \syz^2k)$.
Using \cite[Corollaries 1.10 and 1.15]{LW}, we see that $k$ is a direct summand of $\syz^2k$.
It follows from \cite[Theorem 4.1]{burch} that the local ring $R$ is Burch.
Therefore, the claim reduces to (a) in (4), which is shown just below.

(4) To show assertion (a), we may assume that the local ring $R$ is singular.
Fix a nonzero object $X$ of $\ds(R)$.
In view of Remark \ref{21}(1), there is an $R$-module $M$ with $\pd_RM=\infty$ such that $X$ is isomorphic in $\ds(R)$ to some shift of $M$.
According to \cite[Lemma 7.4]{burch}, there exists an exact sequence $0\to(\syz M)^{\oplus n}\to N\to M^{\oplus n}\to0$ with $n>0$, $\I_1(N)=\m$ and $\pd_RN=\infty$.
By \cite[Proposition 4.2]{burch} the residue field $k$ is a dierct summand of $\syz^2N$.
Hence, $k$ is in $\langle M\rangle_2=\langle X\rangle_2$ in $\ds(R)$.
Thus we obtain $\dx(R)\le1$, and (a) follows.

In what follows, we prove (b).
First, we assume that $R$ is complete.
By definition, there is an $R$-sequence $\xx$ such that $R/(\xx) \cong S/I$, where $S$ is a regular local ring and $I$ is a Burch ideal of $S$.
Let $M$ be a nonzero object of $\ds(R)$.
Using Remark \ref{21}, we may assume that $M$ is an $R$-module of infinite projective dimension and that $\xx$ is an $M$-sequence.
The $R/(\xx)$-module $M/\xx M$ is nonfree by \cite[Lemma 1.3.5]{BH}.
Applying Lemma \ref{227}(2a) to $M/\xx M$, we find an element $\overline y\in R/(\xx)$ such that $k$ belongs to $\langle C\rangle_1^{\ds(R/(\xx))}$, where $C=\k(\overline y,M/\xx M)$.
In other words, we have $\level^C_{\ds(R/(\xx))}(k)\le 0$.
Applying Lemma \ref{227}(2b) to the sequence $\xx, y$ in $R$, we see that $\level^M_{\ds(R)}(k)\le (u+2)(\level_{\ds(R/(\xx))}^C(k)+1)-1=u+1$.
This shows the inequality $\dx(R) \le u+1$.

Now we consider the case where $R$ is not necessarily complete.
We may assume $R$ is not regular.
By \cite[Proposition 4.3]{udim}, the module $\syz^uk$ is in $\add(R\oplus \syz^{u+2}k)$.
Apply (b) in (3) to get $\dx(R)\le(h+1)(2u+4)-1$.

(5) Take an $R$-sequence $\xx=x_1,\dots,x_n$ such that $R/(\xx)$ is a nontrivial fiber product over $k$.
Assertion (1a) implies that $\dx(R/(\xx))\le0$.
Applying Lemma \ref{227}(3b), we get $\dx(R)\le(n+1)(0+1)-1=n$.
\end{proof}

As a corollary of Proposition \ref{t220}(4b), we are able to recover \cite[Lemma 4.18]{BFK}.

\begin{cor}[Ballard--Favero--Katzarkov]\label{34}
If $R$ is a complete local hypersurface, then $\dx(R)\le \dim R+1$.
\end{cor}

Here are a couple of facts to mention concerning the above proposition and corollary.

\begin{rem}\label{23}
\begin{enumerate}[(1)]
\item
Sj\"odin \cite{Sj} proves that if $(R,\m)$ is a Gorenstein local ring with $\m^3=0\ne\m^2$ and $e=\edim R\ge2$, then $\P_k(t)=\frac{1}{1-et+t^2}$; this also follows from the results of Avramov, Iyengar and \c{S}ega \cite[Theorems 4.1 and 4.6]{AIS}, which say that such a ring $R$ is {\em Koszul} in the sense of Herzog and Iyengar \cite{HI}.
Thus, in the situation of Proposition \ref{t220}(1i), if $R$ is Gorenstein, then $R$ is a hypersurface or $\m^2=0$.
\item
The upper bounds of dominant indices obtained in \cite[Corollary 5.5]{udim} are improved by (2) and (3) of Proposition \ref{t220}.
Note that, as we explain in the introduction, a Burch ring and a quasi-fiber product ring are known to be uniformly dominant, and explicit upper bounds are given in \cite[Corollary 5.5]{udim}.
Parts (4) and (5) of Proposition \ref{t220} improve these upper bounds.
\item
Corollary \ref{34} is the same as \cite[Lemma 4.18]{BFK} except that the latter assumes that the ring contains an algebraically closed field of characteristic zero, which does not seem to be used there.
\end{enumerate}
\end{rem}

We record an example of an artinian non-Gorenstein Burch ring that has dominant index exactly one.

\begin{ex}\label{32}
Let $k$ be a field and put $R=k[x,y]/(x^2,xy^2,y^4)$.
Then $R$ is an artinian non-Gorenstein Burch ring with $\Burch(R)=1$ by Example \ref{ex34}.
Consider the cyclic $R$-module $M=R/(x,y^2)$.
We see that $\syz M\cong M^{\oplus 2}$ and that $k$ is not a direct summand of $\syz^iM$ for all $i\ge0$; see \cite[Example 4.5]{DE} again.
Since $\edim R>1$, the ideal $\I_1(\syz^r k)$ coincides with the maximal ideal $\m$, thanks to the construction of a minimal free resolution of $k$ due to Tate \cite{Tate} and Gulliksen \cite{G}; see \cite[pp.\,5--6]{CCHK}.
By \cite[Proposition 4.2]{burch} the residue field $k$ is a direct summand of $\syz^2(\syz^r k)$.
In view of this and the fact that $k$ is indecomposable, we observe that $\syz^r k$ is not a direct summand of $\bigoplus_{i=0}^s\syz^iM^{\oplus n_i}$ for all $r,s,n_0,\dots,n_s\ge0$.
By the claim and Proposition \ref{218}, we see that $\dx(R)>0$.
Using Proposition \ref{t220}(4a), we get $\dx(R)=1$.
\end{ex}

The next proposition contains upper bounds for the dominant index of a Cohen--Macaulay local ring with small multiplicity.
Assertion (4) of the proposition below supports \cite[Question 9.2]{dlr} in the affirmative, which asks whether or not Tor/Ext-friendliness implies dominance.

\begin{prop}\label{p41}
Put $d=\dim R$, $c=\codim R$ and $r=\r(R)$.
Assume either that $d=0$ or that $R$ is Cohen--Macaulay with $k$ infinite.
Then the following statements hold true.
\begin{enumerate}[\rm(1)]
\item The local ring $R$ is a Burch ring with $\dx(R)\le d$ in each of the following cases:\\
{\rm(a)} $\e(R) \le 3$.\qquad
{\rm(b)} $R$ has minimal multiplicity (i.e., $\e(R)=c+1$).\qquad
{\rm(c)} $c=2$ and $\e(R) < r(r+1)$.\\
{\rm(d)} $R$ is non-Gorenstein with $\e(R) \le 5$.
\item The local ring $R$ is a Burch ring with $\dx(R)\le d+1$ in each of the following cases:\\
{\rm(a)} $R$ is non-Gorenstein and G-regular with $\e(R)=6$.\qquad
{\rm(b)} $c=2$ and $\e(R) < 2r(r+1)$.\\
{\rm(c)} $R$ is non-Gorenstein with $c=2$ and $\e(R)\le 11$.
\item
If $\e(R)\le c+2$ and $\P_k^R(t)\ne\frac{(1+t)^d}{1-ct+rt^2}$, then the local ring $R$ is uniformly dominant such that $\dx(R)\le d$.
\item
If $R$ is non-Gorenstein with $\e(R)\le6$, then the following conditions on $R$ are equivalent:\\
{\rm(a)} uniformly dominant,\quad
{\rm(b)} dominant,\quad
{\rm(c)} Tor-friendly,\quad
{\rm(d)} Ext-friendly,\quad
{\rm(e)} G-regular.
\end{enumerate}
\end{prop}

\begin{proof}
Putting $e=\edim R$, we have $c=e-d$.
By assumption, there is a system of parameters $\xx=x_1,\dots,x_d$ of $R$ with $(\xx)$ a reduction of $\m$.
Then $\e(R)=\ell(R/(\xx))$ and $\codim R/(\xx)=c$ by \cite[Proposition 8.3.3(2)]{HS}.

(1) From the definition of a Burch ring we see that if $R/(\xx)$ is Burch ring, then so is $R$.
In view of Lemma \ref{227}(3), replacing $R$ with $R/(\xx)$, we may assume $d=0$, and hence $\e(R)=\ell(R)$.
If $R$ is a field, then it is Burch and $\dx(R)=-1$.
In what follows, we assume that $R$ is not a field.
In case (b), it holds that $\m^2=0$, and hence $R$ is Burch with $\dx(R)=0$ by Remark \ref{r24}(1) and Proposition \ref{t220}(1d).
Cases (c) and (d) are consequences of Theorem \ref{t212}(2), Corollary \ref{18}(2) and Proposition \ref{t220}(1e).
Now we consider case (a).
As $\ell(R)\le 3$, we see that $R$ is isomorphic to $A/(z^3)$ or $A/(z^2)$ or $B/(v^2,vw,w^2)$, where $A$ is a discrete valuation ring with a uniformizer $z$, and $B$ is a regular local ring with a regular system of parameters $v,w$.
The ring $R$ is a hypersurface in the first and second cases, while $R$ has minimal multiplicity in the second and third cases.
Remark \ref{r24} shows that $R$ is Burch.
Also, in the first case, $R$, $R/zR=k$ and $R/z^2R=R/\soc R$ are all the nonisomorphic indecomposable $R$-modules.
We see that $k\in\langle X\rangle_1$ in $\ds(R)$ for all $0\ne X\in\ds(R)$, and hence $\dx(R)=0$.
In the second and third cases, we have $\m^2=0$.
Proposition \ref{t220}(1d) shows $\dx(R)=0$.

(2) Corollaries \ref{18}(1), \ref{20} and Theorem \ref{t212}(1) show $R/(\xx)\cong S/I$, where $S$ is a regular local ring and $I$ is a Burch ideal.
Hence $R$ is Burch.
Moreover, the proof of Proposition \ref{t220}(4b) works, and we get $\dx(R)\le d+1$.

(3) It holds that $\e(R)+d=\ell(R/(\xx))+\ell((\xx)/\xx\m)=\ell(R/\m)+\ell(\m/\m^2)+\ell(\m^2/\xx\m)=1+e+\ell(\m^2/\xx\m)$.
Since $\e(R)\le c+2=e-d+2$, we get $\ell(\m^2/\xx\m)\le1$, so $\m(\m^2/\xx\m)=0$ and $(\m/(\xx))^3=0$.
In view of \cite[Proposition 8.3.3(2)]{HS}, there exist elements $y_1,\dots,y_{e-d}\in\m$ such that $\m=(x_1,\dots,x_d,y_1,\dots,y_{e-d})$.
We see that $\edim R/(x_1,\dots,x_i)=e-i$ and $x_{i+1}\notin(\m/(x_1,\dots,x_i))^2$ for each $0\le i\le d$.
Applying \cite[Proposition 3.3.5(1)]{A} repeatedly, we observe that $\P_k^{R/(\xx)}(t)\ne\frac{1}{1-ct+rt^2}$.
Note that $\edim R/(\xx)=c$ and $\r(R/(\xx))=r$.
Then the inequality $\dx(R/(\xx))\le0$ follows from Proposition \ref{t220}(1i).
Finally, using Lemma \ref{227}(3b), we obtain the desired inequality $\dx(R)\le(d+1)(0+1)-1=d$.

(4) The implications (a)$\Rightarrow$(b) and (d)$\Rightarrow$(e) are obvious.
The implications (b)$\Rightarrow$(c)$\Rightarrow$(d) hold by \cite[Theorem 1.1(5)]{dlr} and \cite[Proposition 5.5]{AINS}.
The implication (e)$\Rightarrow$(a) follows from (2) and (3).
\end{proof}

\begin{rem}
The assumptions in (1a) and (1d) of Proposition \ref{p41}(1a) that $\e(R)\le3$ and that $R$ is non-Gorenstein respectively are necessary.
Indeed, let $R=k[x,y]/(x^2,y^2)$ with $k$ a field.
Then $\e(R)=\ell(R)=4$.
As $R$ is a complete intersection that is not a hypersurface, it is not dominant by \cite[Theorem 1.1(4)]{dlr}.
\end{rem}

We seek for an upper bound of the dominant index of a Cohen--Macaulay local ring of finite representation type in the proposition given after recalling the definition below.

\begin{dfn}
\begin{enumerate}[(1)]
\item
A Cohen--Macaulay local ring $R$ is said to have {\em finite representation type} if there exist only finitely many isomorphism classes of indecomposable maximal Cohen--Macaulay $R$-modules.
\item
For an $R$-module $M$, we denote by $\add M$ the subcategory of $\mod R$ consisting of direct summands of finite direct sums of copies of $M$.
\item
When $R$ is Cohen--Macaulay, $\ocm(R)$ stands for the subcategory of $\mod R$ consiting of modules $M$ that fits into an exact sequence $0\to M\to F\to N\to0$ in $\mod R$ with $F$ free and $N$ maximal Cohen--Macaulay.
\end{enumerate}
\end{dfn}

\begin{prop}\label{14}
Let $(R,\m,k)$ be a $d$-dimensional Cohen--Macaulay local ring.
The assertions below hold.
\begin{enumerate}[\rm(1)]
\item
\begin{enumerate}[\rm(a)]
\item
Let $d=0$.
If $R$ has finite representation type, then $R$ is uniformly dominant with $\dx(R)\le1$.
\item
Let $d=1$ and $|k|=\infty$.
If $R$ has finite representation type, $R$ is uniformly dominant with $\dx(R)\le1$. 
\item
Let $d=2$, and assume that $k$ is algebraically closed.
Let $R$ be an excellent henselian normal domain.
Suppose that there exists an $R$-module $M$ such that $\ocm(R)=\add M$ (this assumption is satisfied if $R$ has finite represenation type).
Then $R$ is a uniformly dominant local ring such that $\dx(R)\le 2$.
\end{enumerate}
\item
Assume that $R$ is an excellent henselian local ring with $d\le2$ and such that $k$ is algebraically closed.
If $R$ has finite representation type, then $R$ is a uniformly dominant local ring such that $\dx(R)\le 2$.
\end{enumerate}
\end{prop}

\begin{proof}
(1a) By \cite[Theorem 3.3]{LW} the ring $R$ is a hypersurface.
Apply Proposition \ref{t220}(4a) and Remark \ref{r24}(2).

(1b) The ring $R$ is analytically unramified by \cite[Proposition 4.15]{LW}, and in particular, $R$ is reduced.
The integral closure $\overline R$ of $R$ in its total quotient ring is module-finite over $R$ by \cite[Theorem 4.6(i)]{LW}.
It follows from \cite[Theorem A.29(iv)]{LW} that $\e(R)=\nu_R(\overline R)$.
If $\nu_R(\overline R)\ge4$, then $R$ has infinite representation type by \cite[Theorem 4.2]{LW}, a contradiction.
We must have $\e(R)\le3$.
Proposition \ref{p41}(1a) implies that $\dx(R)\le1$.

(1c) Under the assumption, $R$ has a rational singularity by \cite[Corollary 3.3]{ncr}, and hence $R$ has minimal multiplicity by \cite[Corollary 6]{Ar}; see also \cite[Proposition 3.8]{ST}.
Proposition \ref{p41}(1b) yields $\dx(R)\le 2$.

(2) When $d=0$, it follows from (1a) that the inequality $\dx(R)\le1$ holds.
Since $R$ has finite representation type, it has an isolated singularity by \cite[Theorem 7.12]{LW}.
Therefore, when $d=2$, the ring $R$ is normal, and the assertion is a special case of (1c).
As $R$ is excellent,the completion $\widehat R$ has an isolated singularity as well.
Hence when $d=1$, the ring $R$ is analytically unramified, and we get the inequality $\dx(R)\le1$ by (1b).
\end{proof}

\begin{ques}
Under the situation of Proposition \ref{14}(1c), can one prove that the inequality $\dx(R)\le1$ holds?
If so, the inequality in the conclusion of Proposition \ref{14}(2) can also be refined to $\dx(R)\le1$.
\end{ques}

\begin{rem}
The Eisenbud--Herzog conjecture \cite[Conjecture 7.21]{LW} asserts that if $R$ is a Cohen--Macaulay local ring with $\dim R\ge2$ and has finite represenation type, then $R$ has minimal multiplicity.
Hence, if the conjecture is true, combining it with Propositions \ref{14}(2) and \ref{p41}(1b) will imply that every excellent henselian Cohen--Macaulay local ring of finite representation type with algebraically closed residue field is uniformly dominant with dominant index at most $\dim R$.

In fact, there are only two known examples of a Cohen--Macaulay non-hypersurface local ring of finite represenation type, both of which have dimension three and satisfy the Eisenbud--Herzog conjecture (i.e., have minimal multiplicity); see \cite[\S7.1]{S}.
Thus, all the known examples of a Cohen--Macaulay local ring $R$ of finite representation type are uniformly dominant with dominant index at most $\dim R$.
\end{rem}

In the rest of this section, we consider the dominant index of a Cohen--Macaulay local ring of codimension two which is not necessary Burch.
We prepare a lemma to show the existence of a certain general element.

\begin{lem}\label{2}
Let $(R,\m,k)$ be a local ring.
Then the following statements hold true.
\begin{enumerate}[\rm(1)]
\item
Let $U$ be a subset of $R$ such that for any two distinct elements $u,v\in U$ the element $u-v$ is a unit of $R$.
Let $a,b\in R$ and assume that $a$ is $R$-regular.
Put $f(t)=a-bt\in R[\![t]\!]$.
Let $\p_1,\dots,\p_n$ be the associated prime ideals of $R$.
For every $x\in\m\setminus\m^2$ it holds that $\#\{u\in U\mid f(ux)\in\bigcup_{i=1}^n\p_i\}\le n$.
\item
Suppose that $k$ is infinite.
Then there exists an infinite set of units of $R$ such that for any two distinct elements $u,v\in U$ the element $u-v$ is again a unit of $R$.
\end{enumerate}
\end{lem}

\begin{proof}
(1) Fix $1\le i\le n$.
Let $u\in U$ and assume $f(ux)\in\p_i$.
We have $f(ux)=a-bux$, while $a\notin\p_i$ as $a$ is $R$-regular.
Hence $bx\notin\p_i$.
For every element $v\in U$ with $v\ne u$, we get $f(vx)-f(ux)=(u-v)bx\notin\p_i$ since $u-v$ is a unit by the assumption on $U$.
It follows that $f(vx)\notin\p_i$.
This argument shows that $\#\{u\in U\mid f(ux)\in\p_i\}\le1$, and we observe that $\#\{u\in U\mid f(ux)\in\bigcup_{i=1}^n\p_i\}\le n$.

(2) Let $U$ be the set of complete representatives in $S$ of the nonzero elements of the residue field $S/\n=k$.
Then by the assumption that $k$ is infinite, we observe that $U$ is such a set as in the assertion.
\end{proof}

The following theorem is the main result of this section.
Among other things, a Cohen--Macaulay local ring of codimension two with an infinite residue field is uniformly dominant unless it is a complete intersection.

\begin{thm} \label{t614}
\begin{enumerate}[\rm(1)]
\item
Let $(R,\m,k)$ be an artinian local ring with $\edim R\le2$ and $k$ infinite.
Then $R$ is either a complete intersection with $\edim R=2$, or a uniformly dominant ring with $\dx(R)\le5$.
\item
Let $(R,\m,k)$ be a Cohen--Macaulay local ring with $\dim R=d$, $\codim R\le2$ and $|k|=\infty$.
Then $R$ is either a complete intersection with $\codim R=2$, or uniformly dominant with $\dx(R)\le6d+5$.
\end{enumerate}
\end{thm}

\begin{proof}
(1) If $\edim R\le1$, then $R$ is an artinian hypersurface, so that $\dx(R)\le 1$ by Proposition \ref{t220}(4a) and Remark \ref{r24}(2).
We may assume that $\edim R=2$, and may also assume that $R$ is not a complete intersection.
There exist a complete regular local ring $(S,\n,k)$ of dimension two with $k$ infinite and an ideal $I$ of $S$ such that $R\cong S/I$.
Then $\height I=\grade I=2$ and $\pd I=1$.
The Hilbert--Burch theorem \cite[Theorem 1.4.17]{BH} implies that the $S$-module $I$ has a minimal free resolution $0\to S^{\oplus r}\xrightarrow{A}S^{\oplus(r+1)}\to I\to0$, where $A=(a_{ij})$ is an $(r+1)\times r$ matrix over $S$ such that $\I_r(A)=I$.
Since $R$ is not a complete intersection, we have $r\ge2$.
If $\I_1(A)$ is not contained in $\n^2$, then the artinian local ring $R$ is Burch by Proposition \ref{211}(1), so that $R$ is uniformly dominant with $\dx(R)\le3$ by \cite[Corollary 5.5(2)]{udim}.
We may assume that $\I_1(A)$ is contained in $\n^2$, and hence $a_{ij}\in\n^2$ for all $i,j$.

For $1\le i\le r+1$ let $\Delta_i$ be the determinant of the $r\times r$ matrix given by removing the $i$th row of the matrix $A$.
As $\I_r(A)=I\ne0$, we have $\Delta_l\ne0$ for some $1\le l\le r+1$.
Applying suitable elementary operations to the matrix $A$ if necessary, we may assume that $l=1$.
Then $f:=\Delta_1\in I$ is $S$-regular.
Choose an element $g\in I$ such that $f,g$ is an $S$-sequence.
Write $g=\sum_{i=1}^{r+1}g_i\Delta_i$.
Let $S[\![t]\!]$
be the formal power series ring.
Put $a_{11}(t)=a_{11}-t$ and $a_{ij}(t)=a_{ij}$ for all $(i,j)\ne(1,1)$.
Let $A(t)=(a_{ij}(t))$ be the $(r+1)\times r$ matrix over $S[\![t]\!]$.
For $1\le i\le r+1$ let $\Delta_i(t)$ be the determinant of the $r\times r$ matrix given by removing the $i$th row of $A(t)$.
Note that $\Delta_1(t)=\Delta_1=f$.
Set $g(t)=\sum_{i=1}^{r+1}g_i\Delta_i(t)$.
Note that for each $2\le i\le r+1$, we have
$$
\Delta_i(t)=\Delta_i-t\,\Gamma_i,\qquad\text{where }\Gamma_i=\left|\begin{smallmatrix}
a_{22}&\cdots&a_{2r}\\
\cdots&\cdots&\cdots\\
a_{i-1,2}&\cdots&a_{i-1,r}\\
a_{i+1,2}&\cdots&a_{i+1,r}\\
\cdots&\cdots&\cdots\\
a_{r+1,2}&\cdots&a_{r+1,r}
\end{smallmatrix}\right|.
$$
Putting $\xi=\sum_{i=2}^{r+1}g_i\Gamma_i\in S$, we have $g(t)=g-t\xi$.
We have $\overline{\Delta_i(t)}=\Delta_i(0)=\Delta_i$ for any $1\le i\le r+1$ and $\overline{g(t)}=g(0)=g$ in $S[\![t]\!]/(t)=S$.
Since $f,g$ is an $S$-sequence, we see that $t,f,g(t)$ is an $S[\![t]\!]$-sequence, so that $f,g(t)$ is an $S[\![t]\!]$-sequence (as the ring $S[\![t]\!]$ is local).
Hence $\grade\I_r(A(t))\ge2$.
By the Hilbert--Burch theorem again, the ideal $\I_r(A(t))$ of $S[\![t]\!]$ has a minimal free resolution $0\to S[\![t]\!]^{\oplus r}\xrightarrow{A(t)}S[\![t]\!]^{\oplus(r+1)}\to\I_r(A(t))\to0$, and $\I_r(A(t))$ is perfect of grade two.
Therefore, the ring $T=S[\![t]\!]/\I_r(A(t))$ is a Cohen--Macaulay local ring of dimension one.
Since $T/(t)=S/\I_r(A)=R$ is an artinian ring, the element $t$ is $T$-regular.

We claim that there exists an element $c\in\n\setminus\n^2$ such that $f,g(c)$ is an $S$-sequence.
Indeed, by Lemma \ref{2}(2), there is an infinite set $U$ of units of $S$ such that for any two distinct elements $u,v\in U$ the element $u-v$ is again a unit of $S$.
Let $x\in\n\setminus\n^2$.
The set $X=\{u\in U\mid g(ux)\in\bigcup_{\p\in\ass_SS/(f)}\p\}$ is finite by Lemma \ref{2}(1).
Since $U$ is an infinite set, we can choose an element $u\in U$ such that $g(ux)\notin\bigcup_{\p\in\ass_SS/(f)}\p$.
The element $g(ux)$ is $S/(f)$-regular, so that $f,g(ux)$ is an $S$-sequence.
Thus the claim follows.

Taking an element $c$ as in the above claim, we get an $S$-sequence $f,g(c)$ in $I':=\I_r(A(c))$.
Applying the Hilbert--Burch theorem to $I'$, we have a minimal free resolution $0 \to S^{\oplus r} \xrightarrow{A(c)} S^{\oplus(r+1)} \to I' \to 0$ of $I'$ over $S$.
As $\dim S=2$, the local ring $R':=S/I'$ is artinian.
Since $a_{11}-c$ is in $\I_1(A(c))\setminus\n^2$, Proposition \ref{211}(1) implies that the artinian local ring $R'$ is Burch.
The isomorphism $R'\cong T/(t-c)$ shows that the complete local ring $T$ is also Burch.
Hence we may apply Proposition \ref{t220}(4b) to obtain the inequality $\dx(T) \le \depth T+1=2$.
Since $T/(t)\cong R$, it follows from Lemma \ref{227}(3a) that $\dx(R)\le 2\dx(T)+1 \le 5$.

(2) By prime avoidance, we choose a system of parameters $\xx=x_1,\dots,x_d$ of $R$ such that $\overline{x_i}\in\overline\m\setminus\overline{\m}^2$ for each $1\le i\le d$, where $\overline\m=\m/(x_1,\dots,x_{i-1})$.
Then $R/(\xx)$ is an artinian local ring with $\edim R/(\xx)=\edim R-d=\codim R\le2$ and with infinite residue field $k$.
It follows from (1) that $R/(\xx)$ is either a complete intersection with $\edim R/(\xx)=2$ or a uniformly dominant ring with $\dx(R/(\xx))\le5$.
In the former case, $R$ is a complete intersection with $\codim R=\edim R/(\xx)=2$.
In the latter case, applying Lemma \ref{227}(3b), we get $\dx(R)\le(d+1)(\dx(R/(\xx))+1)-1\le(d+1)(5+1)-1=6d+5$.
\end{proof}

\begin{rem}
The first assertion of Theorem \ref{t614} recovers and refines \cite[Corollary 6.5]{udim} when the residue field is infinite.
Let $S$ be a regular local ring, $I$ an ideal of $S$ with $\nu(I)\le2$ and $R=S/I$.
Then $\codim R
=\edim R-\dim R
\le\dim S-(\dim S-\height I)
=\height I
\le\nu(I)
\le2$.
Thus, when $R$ is a Cohen--Macaulay local ring with infinite residue field, the second assertion of Theorem \ref{t614} recovers and refines \cite[Corollary 6.11]{udim}.
\end{rem}

\section{Gorenstein local rings with almost minimal multiplicity}\label{27}

The purpose of this section is to pose and consider the following question.

\begin{ques}\label{22}
Let $R$ be a Gorenstein local ring with maximal ideal $\m$.
Suppose that $R$ is not a complete intersection.
If either $\m^3=0$ or $\e(R)\le\codim R+2$, is then $R$ necessarily dominant?
\end{ques}

The second case of Question \ref{22} is viewed as a higher-dimensional version of the first.
Note that Question \ref{22} and \cite[Question 9.6]{dlr} are located in the opposite direction; the latter question asks whether every Gorenstein dominant local ring is a hypersurface.
We start by a lemma that would be well-known to experts.

\begin{lem}\label{24}
\begin{enumerate}[\rm(1)]
\item
Let $R$ be a Gorenstein local ring with maximal ideal $\m$.
Suppose that $\m^3=0$ and that $R$ is not a complete intersection.
Then one has that $\edim R\ge3$ and $\soc R=\m^2\cong k$.
\item
Let $R$ be a Cohen--Macaulay local ring with dimension $d$.
Let $\xx=x_1,\dots,x_d$ be a system of parameters of $R$.
Let $\X$ be a subcategory of $\ds(R/(\xx))$.
Then $\eta(\thick_{\ds(R/(\xx))}\X)$ is contained in $\thick_{\ds(R)}\eta(\X)$, where $\eta:\ds(R/(\xx))\to\ds(R)$ stands for the natural exact functor.
\end{enumerate}
\end{lem}

\begin{proof}
(1) If $\edim R\le2$, then $R$ is a Gorenstein local ring with $\codim R\le2$, so that it is a complete intersection.
Hence $\edim R\ge3$.
If $\m^2=0$, then $R$ has minimal multiplicity, so that it is a hypersurface with $\e(R)=2$.
Hence $\m^2\ne0$.
Since $\m^3=0$, we have $0\ne\m^2\subseteq\soc R\cong k$.
Therefore, $\soc R=\m^2\cong k$.

(2) Put $\Y=\thick_{\ds(R)}\eta(\X)$.
Let $\ZZ=\eta^{-1}(\Y)$, that is, the subcategory of $\ds(R/(\xx))$ consisting of objects $Z$ with $\eta(Z)\in\Y$.
Since $\eta(\X)\subseteq\Y$, we have $\X\subseteq\ZZ$.
Let $L\to M\to N\rightsquigarrow$ be an exact triangle in $\ds(R/(\xx))$.
Then we get an exact triangle $\eta(L)\to\eta(M)\to\eta(N)\rightsquigarrow$ in $\ds(R)$.
Let $\{A,B,C\}=\{L,M,N\}$.
If $A,B\in\ZZ$, then $\eta(A),\eta(B)\in\Y$, and $\eta(C)\in\Y$, so that $C\in\ZZ$.
Also, if $D\oplus E\in\ZZ$, then $\eta(D)\oplus\eta(E)=\eta(D\oplus E)\in\Y$, and $\eta(D),\eta(E)\in\Y$, so that $D,E\in\ZZ$.
Therefore, $\ZZ$ is a thick subcategory of $\ds(R/(\xx))$ containing $\X$.
Thus, $\ZZ$ contains $\thick_{\ds(R/(\xx))}\X$, which means that $\eta(\thick_{\ds(R/(\xx))}\X)\subseteq\Y$.
\end{proof}

We need to recall the definitions of a minimal complete resolution and Betti numbers in negative degrees.

\begin{dfn}
Let $(R,\m,k)$ be an artinian Gorenstein local ring.
Let $M$ be an $R$-module.
Then there exists a {\em minimal complete resolution} of $M$, which is defined as an exact sequence
$$
\cdots\xrightarrow{\partial_4}F_3\xrightarrow{\partial_3}F_2\xrightarrow{\partial_2}F_1\xrightarrow{\partial_1}F_0\xrightarrow{\partial_0}F_{-1}\xrightarrow{\partial_{-1}}F_{-2}\xrightarrow{\partial_{-2}}F_{-3}\xrightarrow{\partial_{-3}}F_{-4}\xrightarrow{\partial_{-4}}\cdots
$$
of free $R$-modules such that $\image\partial_i\subseteq\m F_{i-1}$ for all $i\in\Z$ and $\image\partial_0=M$.
For all $i\in\Z$ we define the $i$th {\em Betti number} $\beta_i^R(M)$ of $M$ as the rank of the free $R$-module $F_i$, and the $i$th {\em syzygy} $\syz_R^iM$ of $M$ as the image of the differential map $\partial_i$.
Obviously, for each nonnegative integer $i$, the $i$th Betti number and syzygy of $M$ are the same things as the usual $i$th Betti number and syzygy of $M$ respectively, so there is no confusion.
\end{dfn}

The following proposition supports the first case of Question \ref{22} in the affirmative.
The third assertion says that it is rather rare that no syzygy contains the residue field as a direct summand.

\begin{prop}\label{25}
Let $(R,\m,k)$ be a Gorenstein local ring with $\m^3=0$ which is not a complete intersection.
\begin{enumerate}[\rm(1)]
\item
Let $I$ be a nonzero proper ideal of $R$.
Then the following two statements hold true.
\begin{enumerate}[\rm(a)]
\item
One has $\thick_{\ds(R)}\{R/I,(R/I)^\ast\}=\ds(R)$.
\item
One has $\thick_{\ds(R)}R/I=\ds(R)$ if either $I$ is principal or $R$ has embedding dimension at most $3$.
\end{enumerate}
\item
A thick subcategory $\X$ of $\ds(R)$ coincides with $\ds(R)$ in each of the following three situations.
\begin{enumerate}[\rm(a)]
\item
The subcategory $\X$ is closed under $(-)^\ast$ (i.e., $\X^\ast\subseteq\X$) and contains a nonzero proper ideal of $R$.
\item
The subcategory $\X$ contains a nonzero proper principal ideal of $R$.
\item
The local ring $R$ has embedding dimension at most $3$, and $\X$ contains a nonzero proper ideal of $R$.
\end{enumerate}
\item
Let $M\ne0$ be an $R$-module with no nonzero free summand such that $k$ is not a direct summand of $\syz^iM$ for all $i\in\Z$.
Put $p=\frac{e-\sqrt{e^2-4}}{2}$, $q=\frac{e+\sqrt{e^2-4}}{2}$, $a=\nu_R(M)$ and $b=\r_R(M)$.
Then for all $i\in\Z$ one has
$$
\beta_i^R(M)=\frac{a(q^{i+1}-p^{i+1})-b(q^i-p^i)}{q-p}\ge2\sqrt{\frac{eab-a^2-b^2}{e^2-4}}
\quad\text{and}\quad
p<\frac{a}{b},\frac{b}{a}<q.
$$
\end{enumerate}
\end{prop}

\begin{proof}
(1a) Lemma \ref{24}(1) implies that $\edim R\ge3$ and $\soc R=\m^2\cong k$.
As $I$ is a nonzero ideal of $R$, it contains $\soc R=\m^2$.
Hence $(\m/I)^2=0$ and therefore the artinian local ring $R/I$ has minimal multiplicity.
Note that $(R/I)^\ast$ is isomorphic to the injective hull $\E_{R/I}(k)$ of the $R/I$-module $k$.
The conclusion of the proposition will follow once we show that $k\in\thick_{\ds(R)}\{R/I,(R/I)^\ast\}$.

First, let us consider the case where the factor local ring $R/I$ is not Gorenstein.
In this case, $\E_{R/I}(k)$ is not $R/I$-free.
Since $\E_{R/I}(k)$ is indecomposable as an $R/I$-module, $\syz_{R/I}\E_{R/I}(k)$ is a submodule of a finite direct sum of copies of $\m/I$, which is annihilated by $\m/I$.
Hence $\syz_{R/I}\E_{R/I}(k)$ is a nonzero $k$-vector space.
We thus obtain an exact sequence $0 \to k^{\oplus a}\to(R/I)^{\oplus b}\to(R/I)^\ast\to0$ of $R/I$-modules with $a,b>0$, which shows that $k$ belongs to $\thick_{\ds(R)}\{R/I,(R/I)^\ast\}$.

Next we consider the case where $R/I$ is Gorenstein.
There are isomorphisms $R/I\cong\E_{R/I}(k)\cong(R/I)^\ast\cong0:I$.
Let $x\in0:I$ be the image of $\overline1\in R/I$ by the composition of these isomorphisms.
It follows that $0:I=(x)\ne0$.
As $I\ne0$, the ideal $(x)$ is proper.
We have $I=0:(0:I)=0:x\cong(R/(x))^\ast$ and $I^\ast\cong R/(x)$.

Assume $R/(x)$ is not Gorenstein.
Applying the above argument on $I$ to $(x)$, we find an exact sequence $0\to k^{\oplus a}\to(R/(x))^{\oplus b}\to(R/(x))^\ast\to0$ with $a,b>0$.
Thus $k$ is in $\thick_{\ds(R)}\{I,I^\ast\}=\thick_{\ds(R)}\{R/I,(R/I)^\ast\}$.

Suppose that $R/(x)$ is a Gorenstein local ring.
Then applying the argument on $I$ given at the beginning to $(x)$, we observe that the Gorenstein local ring $R/(x)$ has minimal multiplicity.
Hence $R/(x)$ is a hypersurface local ring of multiplicity two, which implies that it maximal ideal $\m/(x)$ is isomorphic to $k$.
It follows that $\nu_R(\m)\le2$, which contradicts the fact that $\edim R\ge3$.

(1b) We first consider the case where $I$ is principal.
Writing $I=(x)$, we have $(R/I)^\ast\cong0:I=0:x$.
There is an exact sequence $0\to0:x\to R\xrightarrow{x}R\to R/(x)\to0$, which shows that $0:x=R/(x)[-2]\in\thick_{\ds(R)}R/(x)$.
Hence $(R/I)^\ast$ is in $\thick_{\ds(R)}R/I$, and the assertion follows from (1a).

Next we consider the case where $\edim R\le3$.
Lemma \ref{24}(1) implies that $\edim R=3$ and $\soc R=\m^2\cong k$.
We have $\ell(R)=\ell(R/\m)+\ell(\m/\m^2)+\ell(\m^2)=1+3+1=5$.
Note that $0:I$ is also a nonzero proper ideal.

(i) If $\m I=0$, then $I\cong k^{\oplus n}$ for some $n>0$, and applying $\Hom_R(k,-)$ to the exact sequence $0\to I\to R$ shows $n=1$, so that the assertion follows from the previous case.

(ii) We may assume $\m I\ne0$.
As $\m I\subseteq\m^2\cong k$, we get $\m I=\m^2\cong k$.
Thus $\nu(I)=\ell(I)-\ell(\m I)=\ell(I)-1$.

(iii) If $\m(0:I)=0$, then applying argument (i) to the ideal $0:I$ shows that $\thick_{\ds(R)}R/(0:I)$ coincides with $\ds(R)$, and the duality $(-)^\ast:\ds(R)\to\ds(R)$ yields that $\thick_{\ds(R)}(R/(0:I))^\ast$ coincides with $\ds(R)$, so that $\thick_{\ds(R)}R/I$ coincides with $\ds(R)$ as well, since we have $(R/(0:I))^\ast\cong0:(0:I)=I$.

(iv) We may assume that $\m(0:I)\ne0$.
Applying argument (ii) to $0:I$ shows $\nu(0:I)=\ell(0:I)-1$.
As $\ell(0:I)=\ell((R/I)^\ast)=\ell(R/I)$, we have $\nu(I)+\nu(0:I)=(\ell(I)-1)+(\ell(0:I)-1)=\ell(I)+\ell(R/I)-2=\ell(R)-2=3$.
Since $\nu(I)$ and $\nu(0:I)$ are both positive, we see that either of them is equal to $1$.
If $\nu(I)=1$, then $\thick_{\ds(R)}R/I=\ds(R)$ by (1a).
If $\nu(0:I)=1$, then $\thick_{\ds(R)}R/(0:I)=\ds(R)$ by (1a), and similarly as in (iii), applying $(-)^\ast$ gives $\thick_{\ds(R)}R/I=\ds(R)$.

(2) The assertion is an immediate consequence of (1).

(3) Lemma \ref{24}(1) implies that $e\ge3$ and $\soc R=\m^2\cong k$.
It follows from Remark \ref{23}(1) that $\P_k(t)=\frac{1}{t^2-et+1}$.
We see that $\beta_{i+2}(k)-e\beta_{i+1}(k)+\beta_i(k)=0$ for all $i\ge0$.
As $p,q$ are the two roots of the quadratic equation $t^2-et+1=0$, we get $\beta_i(k)=\frac{q^{i+1}-p^{i+1}}{q-p}$ for all $i\ge0$.
Since $k$ is not a direct summand of $M$, we see that $\soc M$ is contained in $\m M$.
Since $M$ has no nonzero free summand, it is a syzygy.
The equalities $\m^2M=0$ and $\soc M=\m M$ follow.
By \cite[Corollary 3.4(1)]{L}, we have $\P_M(t)=\frac{a-bt}{1-et+t^2}$, from which we observe that $\beta_i(M)=a\beta_i(k)-b\beta_{i-1}(k)$ for all $i\ge1$.
It follows that $\beta_i^R(M)=\frac{a(q^{i+1}-p^{i+1})-b(q^i-p^i)}{q-p}$ for all $i\ge0$.
Note that $\nu(M^\ast)=\r(M)=b$ and $\r(M^\ast)=\nu(M)=a$.
Applying the above argument for $M$ to $M^\ast$, we get $\beta_j(M^\ast)=\frac{b(q^{j+1}-p^{j+1})-a(q^j-p^j)}{q-p}$ for all $j\ge0$.
Using the equality $pq=1$,  we obtain
$$
\textstyle\beta_i(M)=\beta_{-i-1}(M^\ast)=\frac{b(q^{-i}-p^{-i})-a(q^{-i-1}-p^{-i-1})}{q-p}=\frac{a(q^{i+1}-p^{i+1})-b(q^i-p^i)}{q-p}
$$
for all $i\le-1$.
Thus we have shown that $\beta_i(M)=\frac{a(q^{i+1}-p^{i+1})-b(q^i-p^i)}{q-p}$ for all integers $i$.

It holds that $(q-p)\beta_i(M)=(aq-b)q^i+(b-ap)p^i$.
Since $e\ge3$, we have $0<p<1<q$.
Therefore, $\lim_{i\to\infty}q^i=\lim_{i\to-\infty}p^i=\infty$ and $\lim_{i\to\infty}p^i=\lim_{i\to-\infty}q^i=0$.
As $\beta_i(M)\ge0$ for every integer $i$, we must have $aq-b>0$ and $b-ap>0$.
Thus $p<\frac{b}{a}<q$.
We get $p=\frac{1}{q}<\frac{a}{b}<\frac{1}{p}=q$.

Note that $(aq-b)q^i,\,(b-ap)p^i$ are positive for all $i\in\Z$.
As $p+q=e$, $pq=1$ and $q-p=\sqrt{e^2-4}$, we get
$$
\begin{array}{l}
\sqrt{e^2-4}\cdot\beta_i(M)=(q-p)\beta_i(M)=(aq-b)q^i+(b-ap)p^i\ge2\sqrt{(aq-b)q^i\cdot(b-ap)p^i}\\
\phantom{\sqrt{e^2-4}\cdot\beta_i(M)=(q-p)\beta_i(M)=(aq-b)q^i+(b-ap)p^i}=2\sqrt{(aq-b)(b-ap)}=2\sqrt{eab-a^2-b^2},
\end{array}
$$
and hence $\beta_i(M)\ge2\sqrt{\frac{eab-a^2-b^2}{e^2-4}}$ for all integers $i$.
\end{proof}

The following theorem is the main result of this section.
This result supports the second case of Question \ref{22} in the affirmative.

\begin{thm}\label{30}
Let $(R,\m,k)$ be a Gorenstein local ring with $k$ infinite, and suppose that $R$ is not a complete intersection.
Let $\X$ be a thick subcategory of $\ds(R)$.
Then $k$ is in $\X$ if either of the following is satisfied.
\begin{enumerate}[\rm(1)]
\item
The inequality $\e(R)\le\codim R+2$ holds, and there exists a nonfree cyclic maximal Cohen--Macaulay $R$-module $X$ such that both $X$ and $X^\ast$ belong to $\X$.
\item
One has $\e(R)\le5$, and $\X$ contains a nonfree cyclic maximal Cohen--Macaulay $R$-module.
\end{enumerate}
\end{thm}

\begin{proof}
(1) As $k$ is infinite, we can choose a system of parameters $\xx=x_1,\dots,x_d$ of $R$ such that $(\xx)$ is a minimal reduction of $\m$.
The argument at the beginning of the proof of Proposition \ref{p41}(3), we see that $(\m/(\xx))^3=0$.
Let $\Y=\thick_{\ds(R/(\xx))}\{X/\xx X,\Hom_{R/(\xx)}(X/\xx X,R/(\xx))\}$.
Note that $X/\xx X$ is a nonfree cyclic $R/(\xx)$-module and that $\Y$ is closed under $\Hom_{R/(\xx)}(-,R/(\xx))$.
Proposition \ref{25}(1) implies that $k$ belongs to $\Y$.
Applying the functor $\eta:\ds(R/(\xx))\to\ds(R)$, we see that $k$ belongs to $\eta(\Y)$, which is contained in $\thick_{\ds(R)}\{X/\xx X,\Hom_{R/(\xx)}(X/\xx X,R/(\xx))\}$ by Lemma \ref{24}(2).
Since $X$ is maximal Cohen--Macaulay, $\xx$ is a regular sequence on $X$.
The Koszul complex on $\xx$ with respect to $X$ is acyclic, which implies that $X/\xx X$ belongs to $\thick_{\ds(R)}X$.
Using \cite[(A.4.21) and (A.4.23)]{C}, we have
$$
\begin{array}{l}
\Hom_{R/(\xx)}(X/\xx X,R/(\xx))
\cong\rhom_{R/(\xx)}(X/\xx X,R/(\xx))
\cong\rhom_{R/(\xx)}(X\lten_RR/(\xx),R/(\xx))\\
\phantom{\Hom_{R/(\xx)}(X/\xx X,R/(\xx))}
\cong\rhom_R(X,R/(\xx))
\cong\rhom_R(X,R)\lten_RR/(\xx)
\cong X^\ast\lten_RR/(\xx)
\cong X^\ast/\xx X^\ast.
\end{array}
$$
Similarly as above, the last term is in $\thick_{\ds(R)}X^\ast$.
Consequently, $k$ belongs to $\thick_{\ds(R)}\{X,X^\ast\}$, and this is contained in $\X$ by assumption.
We conclude that $k$ is in $\X$.

(2) Let $X\in\X$ be a nonfree cyclic maximal Cohen--Macaulay $R$-module.
As $R$ is not a complete intersection but a Gorenstein ring, we have $\codim R\ge3$.
Hence $\e(R)\le5\le\codim R+2$.
Put $e=\edim R$
It follows from (1) and its proof that we find a minimal reduction $(\xx)$ of $\m$, where $\xx=x_1,\dots,x_d$ is a system of parameters of $R$ such that $(\m/(\xx))^3=0$ and $\ell(\m^2/\xx\m)=\e(R)+d-e-1$.
By assumption we have $\e(R)\le5$.
Since $R$ is not a hypersurface, $R$ does not have minimal multiplicity, so that $\ell(\m^2/\xx\m)\ge1$.
It is observed that $e-d\le3$.
By \cite[Proposition 8.3.3(2)]{HS}, we find elements $y_1,\dots,y_{e-d}$ such that $(x_1,\dots,x_d,y_1,\dots,y_{e-d})=\m$. 
Hence $\edim R/(\xx)=\nu(\m/(\xx))\le e-d\le3$.
Note that $X/\xx X$ is a nonfree cyclic $R/(\xx)$-module.
It follows from Proposition \ref{25}(2) that $k$ belongs to $\thick_{\ds(R/(\xx))}X/\xx X$.
Using Lemma \ref{24}(2) and the maximal Cohen--Macaulayness of $X$, we see that $k\in\thick_{\ds(R)}X/\xx X\subseteq\thick_{\ds(R)}X$.
\end{proof}

\begin{rem}
\begin{enumerate}[(1)]
\item
Let $(R,\m,k)$ be a Gorenstein local ring with $\m^3=0$ and $\edim R\ge2$.
Let $M$ be an $R$-module.
Then $k$ is not a direct summand of $\syz^iM$ for all $i\in\Z$ if and only if $k$ is not a direct summand of $M$ and the $R$-modules $M,M^\ast$ are {\em Koszul} in the sense of Herzog and Iyengar \cite{HI}.

Indeed, it follows from \cite[Theorem 4.6]{AIS} that $M$ is Koszul if and only if $(\syz^ik)^\ast$ is not a direct summand of $M$ for all $i>0$, if and only if $\syz^ik$ is not a direct summand of $M^\ast$ for all $i>0$, if and only if $k$ is not a direct summand of $\syz^{-i}(M^\ast)=(\syz^iM)^\ast$ for all $i>0$, if and only if $k^\ast=k$ is not a direct summand of $\syz^iM$ for all $i>0$.
Hence $M^\ast$ is Koszul if and only if $k$ is not a direct summand of $\syz^{-i}M$ for all $i>0$.
\item
In addition to the assumption of Theorem \ref{30}, suppose that $R$ has an isolated singularity.
Then $\X$ coincides with $\ds(R)$ by \cite[Corollary 2.7]{stcm}.
\end{enumerate}
\end{rem}



\end{document}